\documentclass[11pt,twoside,a4paper]{amsart}
\usepackage{amssymb}
\usepackage{a4wide}
\date{\today}

\begin{document}

\title[]
{Poincar\'e-Lelong type formulas and  
Segre numbers}

\def\la{\langle}
\def\ra{\rangle}
\def\End{{\rm End}}

\def\deg{\text{deg}\,}
\def\e{{\epsilon}}

\def\al{\text{\tiny $\aleph$}}
\def\w{\wedge}
\def\db{\bar\partial}
\def\dbar{\bar\partial}
\def\ddbar{\partial\dbar}
\def\loc{{\rm loc}}
\def\R{{\mathbb R}}
\def\C{{\mathbb C}}
\def\w{{\wedge}}
\def\P{{\mathbb P}}
\def\cp{{\mathcal C_p}}
\def\bl{{\mathcal B}}
\def\A{{\mathcal A}}
\def\B{{\mathcal B}}
\def\Cn{\C^n}
\def\ddc{dd^c}
\def\bbox{\bar\square}
\def\G{{\mathcal G}}
\def\od{\overline{D}}
\def\ot{\leftarrow}
\def\loc{\text{loc}}
\def\D{{\mathcal D}}
\def\M{{\mathcal M}}
\def\Hom{{\rm Hom\, }}
\def\codim{{\rm codim\,}}
\def\Tot{{\rm Tot\, }}
\def\Im{{\rm Im\, }}
\def\K{{\mathcal K}}
\def\Ker{{\rm Ker\,  }}
\def\Dom{{\rm Dom\,  }}
\def\Z{{\mathcal Z}}
\def\E{{\mathcal E}}
\def\can{{\rm can}}
\def\O{{\mathcal O}}
\def\Pop{{\mathcal P}}
\def\L{{\mathcal L}}
\def\Q{{\mathcal Q}}
\def\Re{{\rm Re\,  }}
\def\Res{{\rm Res\,  }}
\def\Aut{{\rm Aut}}
\def\L{{\mathcal L}}
\def\U{{\mathcal U}}
\def\Pk{{\mathbb P}}
\def\Ok{{\mathcal O}}
\def\sr{\stackrel}
\def\dv{{\rm div}}
\def\1{{\bf 1}}
\def\La{{\mathcal L}}
\def\rank{{\rm rank \,}}
\def\GZ{{\mathcal {GZ}}}
\def\J{{\mathcal J}}
\def\F{{\mathcal F}}
\def\mult{{\text{mult}}}
\def\k{{\kappa}}
\def\nbh{{neighborhood }}
\def\Ims{{{Im\,}}}
\def\Zk{{\mathbb Z}}
\def\No{{\mathcal N}}

\newtheorem{thm}{Theorem}[section]
\newtheorem{lma}[thm]{Lemma}
\newtheorem{cor}[thm]{Corollary}
\newtheorem{prop}[thm]{Proposition}

\theoremstyle{definition}

\newtheorem{df}[thm]{Definition}

\newtheorem{remark}[thm]{Remark}

\newtheorem{ex}[thm]{Example}

\numberwithin{equation}{section}

\date{\today}

\author{Mats Andersson}

\address{Department of Mathematics\\Chalmers University of Technology and University of Gothenburg\\
S-412 96 G\"OTEBORG\\SWEDEN}

\email{matsa@chalmers.se}



\thanks{The author was
  partially supported by the Swedish Natural Science
  Research Council}

\begin{abstract} 
Let  $E$ and $F$ be Hermitian vector bundles over a
complex manifold $X$ and let $g\colon E\to F$ 
be a holomorphic morphism. We prove a  
Poincar\'e-Lelong type formula with a residue term $M^g$.
The currents $M^g$ so obtained have an expected functorial property.  
We discuss various applications: If $F$ has a trivial holomorphic subbundle
of rank $r$ outside the analytic  set $Z$, then we get  currents with support on $Z$ that
represent the Bott-Chern classes $\hat c_k(E)$ for $k >\rank E-r$.  
We also consider Segre and Chern forms associated with 
certain singular metrics on $E$.
The multiplicities (Lelong numbers) of the various
components of $M^g$  
only depend on the cokernel of the adjoint sheaf morphism $g^*$. This leads
to a notion of distinguished varieties and Segre numbers
of an arbitrary coherent sheaf, generalizing these notions, 
in particular the Hilbert-Samuel multiplicity,  in case of an ideal sheaf.
\end{abstract}

\maketitle

\section{Introduction}
Let $g$ be a non-trivial holomorphic (or meromorphic) section of  a  Hermitian  line bundle
$L\to X$,  $X$ a complex manifold of dimension $n$, and let $[\dv g]$ be the current of integration associated with the 
divisor defined by $g$.  The Poincar\'e-Lelong formula  states that
\begin{equation}\label{plbas}
dd^c\log|g|^2=[\dv g]-c_1(L),
\end{equation}
where $c_1(L)$ is the first Chern form associated with the Chern connection  
on $L$, i.e.,
$c_1(L)=(i/2\pi) \Theta_L$, where $\Theta_L$ is the curvature form.
Here and throughout this paper
$d^c=(i/2\pi)(\dbar-\partial)$; the constant varies in the literature and is chosen here 
so that 
 $dd^c\log|\zeta_1|^2=[\zeta_1=0]$. 
Thus   
the Bott-Chern class $\hat c_1(L)$ determined by $L$ has the current representative
$[\dv g]$ with support on the zero set $Z$ of $g$. This reflects the fact that
$L$ is trivial in the set  $X\setminus Z$.

Various generalizations to sections of, not necessarily holomorphic, higher rank bundles are 
found in, e.g.,  \cite{HL1}, \cite{MS}, and \cite{Apl}.  In \cite{HL2,HL3} is developed,
for a quite general class of smooth bundle morphisms $g\colon E\to F$, a technique to express 
any characteristic form of $E$ or of $F$ as a sum $L+dT$, where $L$ is a current with support on
the singular set $Z$ of $g$ and $T$ is locally integrable.  It is based on a transgression
that roughly speaking deforms the given connection on $E$ or $F$ so that the associated characteristic form 
concentrates on $Z$.

\smallskip
The aim of this paper is to present variants of \eqref{plbas} when
 $g\colon E\to F$ is a holomorphic morphism with equalities modulo $dd^c$-exact terms,
 and to give some applictions.
In case when $E$ is a trivial line bundle such a result was obtained in \cite{Apl}, using transgression relying on
the ideas and results in \cite{BC}.  

In this paper we use a completely different approach that works for $E$ of higher rank.
Let us first
consider again a section $g$ of the line bundle $L\to X$. 
Recall that $c(L)=1+c_1(L)$ is the full Chern form and 
that the Segre form of $L$ is 
$
s(L)=1/c(L)=1-c_1(L)+c_1(L)^2-\cdots.
$
If 
\begin{equation}\label{mg}
M^g=s(L)\w[\dv g]
\end{equation}
and $W^g=s(L)\log|g|^2$, then \eqref{plbas} can be reformulated as 
\begin{equation}\label{plommon}
dd^c W^g=M^g+s(L)-1.
\end{equation}
For a section  $g$ of  a Hermitian vector bundle $F\to X$ with zero set $Z$,  we introduced in
 \cite{A2,A3} the closed current current  
\begin{equation}\label{mg1}
M^g:=\sum_{k=0}^{\infty}  M^g_k
\end{equation}
where $M^g_k$ are the residues 
$
M^g_{k}:=\1_Z  [dd^c \log|g|^2]^{k}, \  k=0,1,2, \cdots, 
$
of the generalized  Monge-Amp\`ere products $ [dd^c \log|g|^2]^{k}$, see Section~\ref{ma-produkt}.
%
The currents $M^g_k$ are {\it generalized cycles}, a notion introduced in \cite{aeswy1}, see Section~\ref{gztrams}.
A generalized cycle $\mu$ of codimension $k$ has well-defined integer multiplicities $\mult_x\mu$ at each point $x$
 and a unique global decomposition into a (Lelong current of a) cycle of codimension $k$, the {\it fixed part},
 and the {\it moving part}; the multiplicities of the latter one vanish outside an analytic set of codimension $\ge k+1$.  
 In case $\mu$ is positive this is the Lelong numbers and its Siu decomposition of $\mu$, respectively. 
It was proved in \cite{aswy, aeswy1} that $\mult_x M^g_k$ coincide with the so-called {\it Segre numbers} of the
ideal $\J_x$ at $x$ generated by $g$, generalizing the Hilbert-Samuel multiplicity of $\J_x$,  
and that the fixed part of $M^g_k$ is the sum  (with multiplicities) of  the {\it distinguished varieties} of the ideal.
See Section~\ref{segresection}. %

\smallskip
Notice that a section $g$ of $F$ is can be considered as a morphism $X\times\C\to F$. 
For an arbitrary holomorphic morphism $g\colon E\to F$, where $E$ and $F$ are Hermitian vector bundles over $X$,
we introduce in this paper a current 
$M^g=M^g_0+\cdots +M^g_n$,
which coincides with $M^g$ above when $E$ is trivial.
The current $M^g$ has
support on the analytic set $Z$ where $g$ is not injective.
Here $M^g_k$ are closed currents of bidegree $(k,k)$ and in fact generalized cycles.
Notice that  
$\Im g$ is a subbundle of $F$ over $X\setminus Z$ and thus the associated Segre form $s(\Im g)$ is defined there,
cf.~Section~\ref{prel}. 
Our first main result  is 

\begin{thm}\label{thmA}
With the notation above  
$ \1_{X\setminus Z}s(\Im g)$ is locally integrable in $X$ and there is a
current $W^g$ with singularities along $Z$ such that
\begin{equation}\label{main1}
dd^c W^g= M^g+  \1_{X\setminus Z}s(\Im g)  -s(E).
\end{equation}
\end{thm}

If $E$ is trivial and $F$ is a line bundle, then \eqref{main1} is precisely \eqref{plommon}.
In case $E$ is a line bundle Theorem~\ref{thmA} as well as other results in this paper 
are readily deduced from \cite{Apl} combined with \cite{aeswy1}, cf.~Remark~\ref{apl}.
The substantial novelty therefore is when $E$ has higher rank. 

\smallskip
Further properties of $W^g$ and $M^g$  are stated in Theorem~\ref{thmett}.  For instance, the multiplicities
$\mult_xM_k$ are non-negative, and independent of the metrics on $E$ and $F$.
Moreover,  $M^g$ satisfy a certain functorial property so that its definition is determined
by the case when $g$ is generically an isomorphism. Then 
$Z$ has positive codimension (unless $X$ is a point) and thus $\1_{X\setminus Z} s(\Im g)=s(F)$.
We have  the following direct generalization of \eqref{plommon}. 

\begin{cor}\label{ansgar1}
If  $g\colon E\to F$ is generically an isomorphism,  then 
 \begin{equation}\label{main111}
dd^c W^g=M^g +s(F)-s(E).
\end{equation}
 \end{cor}
 We have variants of  \eqref{plbas};  notice that $c(F)\w M^g$ is
a current with support on $Z$. 

\begin{prop}\label{avig}
 Assume that $E$ is trivial with trivial metric and $g$ is generically injective.
 Then $\1_{X\setminus Z} c(F/\Im g)$
 is a locally integrable closed current in $X$ and 
 there is a current $V^g$ with singularities  along $Z$ such that  
\begin{equation}\label{main123}
dd^c V^g
=c(F)\w M^{g}+
\1_{X\setminus Z}c(F/\Im g)-c(F). 
\end{equation}
\end{prop}

Since $c_k(F/\Im g)=0$ in $X\setminus Z$ for $k>m-r$, $r=\rank E$,
we get from \eqref{main123}:

\begin{cor}\label{cor123}
If $g_1,\ldots,g_r$ are sections of $F$ that are linearly independent outside $Z$, 
then there are currents $V_k^g$ such that 
\begin{equation}\label{repet}
dd^c  V_{k-1}^g =(c(F)\w M^{g})_k -c_k(F), \quad k>m-r.
\end{equation}
\end{cor}

The $g_i$ define a  trivial subbundle $F$ of rank $k$ in $X\setminus Z$. It is therefore expected  
that the Bott-Chern classes $\hat c_k(F)$, $k>m-r$, can be represented by  currents that have support on 
on $Z$.  In case $r=1$, Corollary~\ref{cor123} appeared in \cite{Apl},
cf.~Remark~\ref{apl} below.

\smallskip

We now turn our attention to a slightly different generalization of the Poincar\'e-Lelong formula. 
Assume that $g\colon E\to F$ is a morphism  as before and that $g$ has optimal rank on $X\setminus Z_0$. 
In this open set we have the short exact sequence $0\to \Ker g\to E\to \Im g\to 0$ and hence
the (non-isometric) isomorphism
$
a\colon E/\Ker g \simeq  \Im g. 
$
Therefore there is a smooth form $w$ in $X\setminus Z_0$ such that 
$dd^c w= s(\Im g)-s(E/\Ker g).$ 
We have an extension across $Z_0$:

\begin{thm}\label{thmB}  
The natural extensions  $\1_{X\setminus Z_0}s(E/\Ker g)$ and $\1_{X\setminus Z_0}s(\Im g)$ are locally integrable 
and closed. There is a current $M^a$  with support on
$Z_0$ and a current  $W^a$ such that 
\begin{equation}\label{trelikhet}
dd^c W^a=M^a+\1_{X\setminus Z_0}s(\Im g) -\1_{X\setminus Z_0}s(E/\Ker g).
\end{equation}
\end{thm}

In Section~\ref{plakat} we give an extended version  (Theorem~\ref{thmtre}). The current $M^a$ is
(at least locally) a generalized cycle and it turns out that $\mult_x M^a_k$ may be negative.  

\smallskip
In Section~\ref{BC-section} we discuss  Chern and Segre forms
associated with some singular metrics on a vector bundle. 
A notion of distinguished varieties and 
Segre type numbers of a general coherent sheaf are discussed in Section~\ref{segresection}.  
In case $Z=\{x\}$ is a point the number $\mult_x M^g_n$ is equal to the so-called Buchsbaum-Rim
multiplicity, \cite{BR}, see Remark~\ref{brjox}.

The plan for the rest of the paper is as follows. 
In Section~\ref{prel} we have collected material that is known, 
except for the regularization in Propositions~\ref{vprop} and \ref{circvprop}. Then we discuss
modifications that admit extensions of certain generically defined subbundles in
Section~\ref{extsub}.  In Sections~\ref{defsec}, \ref{proofsection} and ~\ref{behaviour}
we define $M^g$ and state and prove the main results. 
The proofs rely on results from \cite{aeswy1} and \cite{Kalm},
and are inspired by \cite{LRS,LRSW}.    
A new Siu type result for generalized cycles, proved in Section~\ref{treproofapa},
is crucial for the proof of Theorem~\ref{thmtre}.  
In the last section,  Section~\ref{exsection},  we compute various examples that aim to shed 
light on the notions and results.

\vspace{.5cm}

\section{Preliminaries}\label{prel}

Throughout this paper $X$ is a connected complex manifold of dimension $n$.

\subsection{Singularities of logarithmic type}\label{prel1}
A current $W$ on $X$ is of logarithmic type along the subvariety $Z$, cf.~\cite{BGS}, if $W$ is smooth
in $X\setminus Z$, locally integrable in $X$, and so that the following holds:
Each point on $Z$ has a \nbh $U$ such that $W|_U$ is the direct image under a proper mapping $h\colon \widetilde U\to U$
of a smooth form $\gamma$ in $h^{-1}(U\setminus \Z)$ that locally in $\widetilde U$ has the form 
$\gamma= \sum_j \alpha_j\log|\tau_j|^2 +\beta$,  where $\alpha_j,\beta$ are smooth forms, 
$\alpha_j$ are closed, and $\tau_j$ are local coordinates. 

This requirement is imposed, see, e.g., \cite{BGS,Soule}, to make it possible to
define multiplication of $v$ and the Lelong current of 
another variety intersecting $Z$ properly.  
In this paper we use this notion merely to point out that the current in question has quite simple
singularities.

\subsection{Segre and Chern classes}\label{prel2}
Assume that $\pi\colon E\to X$ is a holomorphic vector bundle, let $\Pk(E)$ be its projectivization
(so that at each point $x\in X$ the fiber consists of all lines through the origin in $E_x$),
and let $p\colon \Pk(E)\to X$ be the natural submersion.  Consider the pullback
$p^*E\to \Pk(E)$ and let $L=\Ok(-1)\subset p^*E$ be the tautological line bundle, equipped with the induced  Hermitian
metric, and Chern form $c(L)=1+c_1(L)$. Then
\begin{equation}\label{segreclass}
s(E)=p_*(1/c(L))=\sum_{k=0}^\infty (-1)^k p_* c_1(L)^k
\end{equation}
and 
\begin{equation}\label{chern0}
c(E)=\frac{1}{s(E)}.
\end{equation}
Since $p$ is a submersion,  $s(E)$ and $c(E)$ are smooth closed forms.
It is proved in \cite{Mour} that this definition of Chern form of $E$ coincides with the differential-geometric definition
\begin{equation}
c(E)= \det \big(I_E+\frac{i}{2\pi}\Theta_E\big),
\end{equation}
where $\Theta_E$ is the  curvature tensor associated with the Chern connection.  

If $h\colon X'\to X$ is a holomorphic mapping, then 
\begin{equation}\label{anka1}
s(h^*E)=h^*s(E), \quad c(h^*E)=h^*c(E).
\end{equation}

If  $g\colon E\to E'$ is an holomorphic vector bundle isomorphism, then we have an induced biholomorphic mapping
$\tilde g\colon \Pk(E)\to \Pk(E')$. If $L'$ is the tautological line bundle over $\Pk(E')$,
then $L=\tilde g^* L'$. If $E$ and $E'$ are Hermitian,  then  there is a smooth form $w$ such that
$
dd^c w=s(E')-s(E).
$

More generally,  
$ 
0\to S\to E\to Q\to 0
$ 
is a short exact sequence of holomorphic Hermitian vector bundles over $X$, then, see \cite{BC},  there is a smooth form
$v$ so that
\begin{equation}\label{chern1}
dd^c v=c(E)-c(Q)c(S).
\end{equation}
It follows from \eqref{chern0} that we have a similar relation for the Segre forms. In fact, if
$w=-s(E)s(Q)s(S) v$, then
$dd^c w=s(E)-s(Q)s(S)$.

\subsection{Generalized cycles}\label{gztrams}
Let $\Z(X)$ be the   $\Zk$-module of analytic cycles on $X$; i.e.,  
locally finite sums 
$$
\sum a_j Z_j,
$$
where $Z_j$ are irreducible subvarieties $Z$ of $X$.  Such a sum can be identified with its Lelong current
$$
\sum a_j [Z_j].
$$

Let $\tau\colon W\to X$ be a proper holomorphic mapping, and let 
$\gamma=c_{k_1}(E_1)\cdots c_{k_\rho}(E_\rho)$ be a product of 
components of Chern forms   of Hermitian vector bundles $E_1,\ldots, E_\rho$ over $W$.
Then
$\tau_*\gamma$ is a closed current of order $0$ on $X$. Let  $\GZ(X)$
be the $\Zk$-module of
all locally finite sums of such currents. If we identify  cycles with their Lelong currents
we get a natural inclusion $\Z(X)\subset \GZ(X)$.  
This module was introduced in
\cite{aeswy1} and all properties stated here can be found there with proofs.

We have a natural decomposition
$$
\GZ(X)=\sum_{k=0}^{\dim X} \GZ_k(X),
$$
where $\GZ_k(X)$ is the elements of dimension $k$; that is, of bidegree $(n-k,n-k)$.
Each generalized cycle has a well-defined Zariski-support.  However the support of $\mu$ can
have strictly larger dimension than the dimension of $\mu$, cf.~Example~\ref{urk}.


Given any analytic variety in $X$ we have the  natural restriction operator
$$
\1_V\colon  \GZ_k(X)\to \GZ_k(X), \quad \mu\mapsto \1_V \mu.
$$
There is a notion of irreducibility and any $\mu\in \GZ_k(X)$ has a 
unique decomposition into irreducible terms. Moreover, $\1_V\mu$ is precisely the sum
of the irreducible components of $\mu$ whose Zariski-supports are contained in $V$. 

If $\gamma$ is a component of a Chern form on $X$, then we have the mapping 
\begin{equation}\label{mulle}
\mu\mapsto \gamma\w \mu
\end{equation}
on $\GZ(X)$.

\smallskip
If $h\colon X\to Y$ is a proper
mapping, then we have a natural mapping $h_*\colon \GZ(X)\to \GZ(Y)$, which is consistent with 
the usual push-forward mapping of cycles. 
One can define $\GZ(Z)$ just as well for a non-smooth reduced analytic space $Z$.  
If $i\colon Z\to X$ is an inclusion, then the image of $i_*$ is precisely the elements
in $\GZ(X)$ that has support on $i_*Z$.   
It is therefore often natural to think of generalized cycles
as purely geometric objects on $X$ and suppress the fact that they formally are currents. 

If  $\mu\in\GZ_k(X)$, then for each point $x\in X$ there is a well-defined integer $\mult_x \mu$, the multiplicity of $\mu$  at $x$. 
If $\mu$ is effective, i.e., a positive current,  it is precisely the Lelong number of $\mu$ at $x$.
It coincides with the usual notion of multiplicity 
if $\mu$ is an analytic cycle.  If $\mu$ is in $\GZ(Z)$ and
$i\colon Z\to X$  is an inclusion,  then $\mult_x \mu=\mult_{i(x)}i_*\mu$.

\begin{ex}\label{urk}
If $X=\Pk^2_{[x_0,x_1,x_2]}$ then $\mu=dd^c\log(|x_1|^2+|x_2|^2)$ is in $\GZ(\Pk^2)$.
It is smooth
except at $p=[1,0,0]$, and $\mult_x\mu=1$ at $x=p$ and $0$ elsewhere. Moreover, $\mu$ is irreducible, 
has dimension $1$, and its Zariski-support is $X$. 
\end{ex}

\smallskip
We say that $\beta$ is a $B$-form on $W$ if it is a component of the form
$c(E)-c(S)c(Q)$, where $0\to S\to E\to Q$ is a short exact sequence
of Hermitian vector bundles on $W$. 
We say that $\mu\sim 0$ in $\GZ_k(X)$ if it is a locally finite sum of currents
of the form $\tau_*(\beta\w \gamma)$,  where $\tau\colon W\to X$ is proper, 
$\beta$  is a $B$-form and
$\gamma$ is a product of components of Chern forms on $W$. 

We let $\mathcal B_k(X)=\GZ_k(X)/\sim$ and $\mathcal B(X)=\oplus_{k=0}^\infty \mathcal B_k(X)$.
It turns out that  $\Z(X)$ is a submodule of $\mathcal B(X)$ as well. The other properties mentioned
above regarding $\GZ(X)$ still hold for $\mathcal B(X)$.
The most important one in this paper is that 
the multiplicity $\mult_x\mu$ of  $\mu\in \GZ_k(X)$ 
only depends on the class of $\mu$ in $\mathcal B_k(X)$.


\begin{lma}\label{nolllma}
If $\gamma$ has positive bidegree, then, cf.~\eqref{mulle}, 
$
\mult_x(\gamma\w \mu)=0.
$
\end{lma}

Any $\mu$ is in $\GZ_{n-k}(X)$
has a unique decomposition
\begin{equation}\label{siu}
\mu=\sum_j \beta_j[Z_j] + N,
\end{equation}
where $Z_j$ have codimension $k$ and $\mult_x N$ vanishes outside an analytic set
of codimension $\ge k+1$. In case $\mu$ is effective, i.e., the $(k,k)$-current $\mu$ is a positive,
then \eqref{siu} is the Siu decomposition of $\mu$.  For a general $\mu$, see Theorem~\ref{musiu}
below.
If $\mu'$ is another representative of the same class in $\mathcal B_{n-k}(X)$,
then the Lelong current in its decomposition \eqref{siu} is the same whereas the term $N$ may be different.  As already mentioned in
the introduction the Lelong current and $N$ are referred to as the fixed and moving part, respectively, of $\mu$.

\subsection{Generalized Monge-Amp\`ere products} \label{ma-produkt}
Let us assume that $X$ is connected and let $\phi$ be a section, with zero set $Z$,  of 
the Hermitian vector bundle $F\to X$.  
One can recursively define closed currents of order zero,  
\begin{equation}\label{ma}
[dd^c\log|\phi|^2]^0=1, \ \  [dd^c\log|\phi|^2]^k=dd^c\big( \log|\phi|^2 \1_{X\setminus Z} [dd^c \log|\phi|^2]^{k-1}\big), \ k=0,1,2, \ldots 
\end{equation}
For each $k\ge 0$  
$$
M^\phi_{k}:=\1_Z  [dd^c \log|\phi|^2]^{k}
$$
is a closed current of order $0$ of bidegree $(k,k)$ with support on $Z$ so it vanishes 
if $k\le \codim Z$. 
Thus \eqref{ma} is the classical Bedford-Taylor-Demailly product for $k\le \codim Z$. 
The definition for larger $k$ might look artificial, but indeed, e.g.,  \cite[Proposition~4.4]{A3},
\begin{equation}\label{maepsilon}
[dd^c\log|\phi|^2]^k=\lim_{\e\to 0} \big(dd^c\log(|\phi|^2+\e)\big)^k, \quad k=0,1,2,\ldots.
\end{equation}

For future reference we sketch a proof for that this definition makes sense:
If $\phi$ is identically $0$ then $M^\phi=M^\phi_0=\1_Z=1$.
Let us assume that $Z$ has positive codimension. 
 Let 
$
\pi\colon \widetilde X\to X
$
be a smooth modification such that the sheaf generated by the section $\pi^*\phi$ of $\pi^*F\to \widetilde X$ is principal,
and generated by the section $\phi^0$ of a line bundle $\La\to \widetilde X$. 
Then\footnote{Such a smooth modification exists by Hironaka's theorem. 
The argument works just as well if one takes $\pi$ as the normalization of the blow up of $X$ along the 
ideal sheaf defined by $\phi$.}
$$
\pi^*\phi=\phi^0 \phi',
$$
where $\phi'$ is a section of $\La^*\otimes \pi^*F $.
Since  
$
\La\to \pi^* F,  \quad  v\mapsto v\phi',
$
is injective, $\La$ is   a subbundle of $\pi^*F$. If we equip $\La$ with the induced metric, then
$|\phi^0|=|\pi^*\phi|$ and 
\begin{equation}\label{tutt4}
dd^c\log|\pi^*\phi|^2= 
dd^c\log|\phi^0|^2=[D]-c_1(\La)= [D]+s_1(\La)
\end{equation}
by \eqref{plbas}.
In particular,
$$
\1_{X'\setminus |D|} dd^c\log|\pi^*\phi|^2     =s_1(\La).
$$
Let  
\begin{equation}\label{flax1}
\la dd^c\log|\phi|^2\ra^\ell:=\1_{X\setminus Z}[dd^c\log|\phi|^2]^\ell=
\pi_* s_1(\La)^\ell,
\end{equation}
\begin{equation}\label{flax2}
\log|\phi|^2\la dd^c\log|\phi|^2\ra^\ell=\pi_*\big(\log|\pi^*\phi|^2 s_1(\La)^\ell\big),
\end{equation}
\begin{equation}\label{pulka1}
[dd^c\log|\phi|^2]^\ell=\pi_*\big([D]\w s_1(\La)^{\ell-1}+ s_1(\L)^\ell\big)
\end{equation}
and
\begin{equation}\label{pulka2}
M^\phi_\ell=\1_{Z}[dd^c\log|\phi|^2]^\ell=\pi_*\big([D]\w s_1(\La)^{\ell-1}).
\end{equation}
It follows that the currents in \eqref{flax1} and \eqref{flax2} are locally integrable. Moreover,
since  $|D|=\pi^{-1}Z$ ($|D|$ is the Zariski-support of $D$), it follows that  
\begin{equation}\label{spargris}
dd^c(\log|\phi|^2\la dd^c\log|\phi|^2\ra^{\ell-1})=[dd^c\log|\phi|^2]^\ell=M_\ell^\phi+\la dd^c\log|\phi|^2\ra^{\ell}, \quad \ell=1,2,\ldots.
\end{equation}
Thus the recursion \eqref{ma} makes sense and produces precisely the currents
 $\la dd^c\log|\phi|^2\ra^\ell$, $[dd^c\log|\phi|^2]^\ell$ and $M^\phi_\ell$.  
 From \eqref{flax1}, \eqref{pulka1} and \eqref{pulka2} we see that  
they are generalized cycles on $X$. 

We let  $M^\phi=M^\phi_0+M^\phi_1+\cdots$.   
If $\pi\colon \widetilde X\to X$ is any modification, then
\begin{equation}\label{struts22}
\pi_* M^{\pi^*\phi}=M^\phi,
\end{equation}
see \cite{aeswy1}. 
Furthermore,  
if $\hat\phi$ is a section of a Hermitian bundle $\hat F\to X$ such that
$|\hat \phi|\sim |\phi|$ locally on $X$, then $M^{\hat \phi}$ and $M^\phi$ define
the same element in  $\mathcal B(X)$.  In particular, $\hat F$ can be 
$F$ but with another Hermitian metric.

\smallskip  

For a thorough discussion of  regularizations of generalized  
Monge-Amp\`ere products, see, e.g., \cite{LSW}. 
We will need the following variant that, as far as we know, has not appeared before.

\begin{prop}\label{vprop}
Let $\chi(t)$ be a smooth function on $\R$ that is $0$ for $t<1/2$ and $1$ for $t>3/4$ and let
$\chi_\e=\chi(|\varphi|^2/\e)$, where
 $\varphi$ is a section of a vector bundle (tuple of holomorphic functions) with zero set $V$ of positive codimension in $X$.
Then the currents\footnote{The first term on the right hand side of \eqref{charm} vanishes unless $\phi\equiv 0$.}
\begin{equation}\label{charm}
T^{\phi,\e}_V= (1-\chi_\e)\1_Z +\dbar\chi_\e\w \frac{\partial\log|\phi|^2}{2 \pi i}\w \sum_{\ell=0}^\infty \la dd^c\log|\phi|^2\ra^\ell 
\end{equation}
tend to $\1_V M^\phi$ when $\e\to 0$.
\end{prop}

\smallskip\noindent
If $V$ contains $Z$ then $T^{\phi,\e}_V$ are smooth and tend to $M^\phi$.

\begin{proof}
It is clear that $(1-\chi_\e)\1_Z\to \1_V\1_Z=\1_V M^\phi_0$. 
Let 
$$
T=\sum_{\ell=0}^\infty \la dd^c\log|\phi|^2\ra^\ell.
$$
We have 
\begin{equation}\label{sardin}
\dbar\chi_\e\w \partial\log |\phi|^2(2 \pi i)^{-1}i\w T=\dbar\big (\chi_\e  \partial\log|\phi|^2(2 \pi i)^{-1}\w T\big)-
\chi_\e\w \dbar\partial\log|\phi|^2(2 \pi i)^{-1}\w T.
\end{equation}
In view of \eqref{flax2}
$$
\chi_\e  \partial\log|\phi|^2 \w T\to \partial\log|\phi|^2\w T,
$$
and hence the first term on the right hand side of \eqref{sardin}, cf.~\eqref{spargris},  tends to 
$$
\sum_{\ell=1}^\infty [dd^c\log|\phi|^2]^\ell,   
$$ 
whereas the second term tends to 
$$
\1_{X\setminus V} \sum_{\ell=1}^\infty [dd^c\log|\phi|^2]^\ell. 
$$
Now \eqref{charm} follows since $\codim V>0$ and $\la dd^c\log|\phi|^2\ra^\ell$ is locally integrable, so that 
$$
\1_V [dd^c\log|\phi|^2]^\ell=\1_V M^\phi_\ell+ \1_V\la dd^c\log|\phi|^2\ra^\ell=\1_V M_\ell^\phi.
$$
\end{proof}

\subsection{Twisting with a line bundle}\label{twistma}
We keep the notation from the previous subsection. Let  $S\to X$ be a line bundle (with no specified metric)
and assume that $\psi$ is a section of $F\otimes S^*$.   If $s$ is a local non-vanishing 
section of $S$ we let $|\psi|_\circ=|s\psi|$. 
Then $dd^c\log|\psi|_\circ:=dd^c\log|s\psi|$ is independent of the choice of $s$ and hence a global current on $X$.
In this way we define the global currents $[dd^c\log|\psi|^2_\circ]^\ell:=
[dd^c\log|s\psi|^2]^\ell$,  cf.~Remark~\ref{kahler} below, and
$
\mathring M^\psi:= M^{s\psi}.
$

\begin{lma}\label{ringlemma}
If $\pi\colon \widetilde X\to X$ is a modification, then
\begin{equation}\label{struts222}
\pi_*\mathring M^{\pi^*\psi}=\mathring M^\psi.
\end{equation}
The current $\mathring M^\psi$ is an element in $\GZ(X)$.
If $\hat \psi$ is a section of $\hat F\otimes S^*$, where $\hat F\to X$ is another Hermitian vector bundle
and $|\hat\psi|\sim|\psi|$, 
then $\mathring M^{\hat \psi}$ defines the same class in $\mathcal B(X)$.
\end{lma}

Notice that $\pi^*\psi$ is a section of $\pi^*F\otimes \pi^*S^*$; we define  
$\mathring M^{\pi^*\psi}$ by suppressing 
$(\pi^*S)^*$. 

\begin{proof}
Since locally  $\mathring M^\psi=M^{s \psi}$, where $s$ is a local non-vanishing section of $S$,
by \eqref{struts22}
$$
\pi_* \mathring M^{\pi^*\psi}=\pi_* M^{\pi^*{s\psi}}=M^{s\psi}=\mathring M^\psi,
$$
and thus \eqref{struts222} holds.

We now choose\footnote{Since the case $\psi\equiv 0$ is trivial,  we may assume that $\psi$ is not vanishing identically.}
$\pi\colon \widetilde X\to X$ such that $\pi^* \psi$ is principal, as in Section~\ref{ma-produkt}. Then 
$\pi^*\psi=\psi^0 \psi'$, where $\psi'$ is a non-vanishing section of $\pi^*F\otimes\L^*\otimes \pi^*S^*$.
Then $\pi^*(s\psi)=(\pi^*s)\psi^0 \psi'$. As in Section~\ref{ma-produkt} we see that 
$s_1(\L)=dd^c\log|(\pi^*s)\psi'|^2=dd^c\log|\psi'|_\circ$, where the $\circ$ means that
$\pi^*S^*$ is suppressed, so that 
\begin{equation}\label{snaps}
\mathring M^{\pi^*\psi}=[D]\w s(\L).
\end{equation}
Hence 
\begin{equation}\label{struts2222}
\mathring M^\psi =\pi_*\mathring M^{\pi^*\psi}=\pi_*([D]\w s(\L))
\end{equation}
is an element in $\GZ(X)$. 

If  $|\hat\psi|\sim|\psi|$, then 
$\pi^*\hat\psi=\psi^0 \hat \psi'$ and therefore, cf.~\eqref{snaps}, 
$\mathring M^{\pi^*\hat \psi}=[D]\w \hat s(\L)$, where
$\hat s(\L)$ denotes the Segre form of $\L$ with respect to the metric induced
by $\hat F$.  Thus $\mathring M^{\pi^*\hat \psi}$ and $\mathring M^{\pi^*\psi}$
are in the same class in $\mathcal B(\widetilde X)$, and so the last part follows. 
\end{proof}

We have the following variant of Proposition~\ref{vprop}.
Let $\chi_\e$ be a sequence ad in this proposition that tends to $\1_{X\setminus V}$. 

\begin{prop}\label{circvprop}
Assume that $\phi$ is a section of $F\otimes S^*$ and that $\alpha$ is a non-vanishing  section of 
$H\otimes S^*$ for some Hermitian vector bundle $H\to X$.   
Then the currents
\begin{equation}\label{charm1}
\mathring T^{\phi,\e}_V= (1-\chi_\e)\1_Z +\dbar\chi_\e\w \frac{\partial\log(|\phi|/|\alpha|)^2}{2\pi i}\w \sum_{\ell=0}^\infty \la dd^c\log|\phi|_\circ^2\ra^\ell 
\end{equation}
tend to $\1_V \mathring M^\phi$ when $\e\to 0$.
\end{prop}

Here $|\phi|/|\alpha|$ is the  global function defined locally as $|s\phi|/|s\alpha|$, where $s$ is any local non-vanishing section of $S^*$. 

\begin{proof}
Given a local section $s$ we have, with the notation in Proposition~\ref{vprop}, that 
\begin{equation}\label{buhu}
2\pi i \mathring T^{\phi,\e}_V= 2\pi i T^{s\phi,\e}_V -\dbar\chi_\e\w \partial\log|s\alpha|^2\w \sum_{\ell=0}^\infty \la dd^c\log|\phi|_\circ^2\ra^\ell.
\end{equation}
Since $\partial\log|s\alpha|^2$ is smooth, letting $T$ denote the last sum,  the last term in 
\eqref{buhu} is equal to 
$$
\dbar \big( \chi_\e \partial\log|s\alpha|^2\w T\big)- \chi_\e \dbar\partial\log|s\alpha|^2\w T
$$
which tends to $\dbar\partial\log|s\alpha|^2\w T-\dbar\partial\log|s\alpha|^2\w T=0$, since
$V$ has positive codimension so that $\1_V T=0$.  By Proposition~\ref{vprop} thus
$\mathring T^{\phi,\e}_V=T^{s\phi,\e}_V+o(1)\to \1_V M^{s\phi}=\1_V\mathring M^\phi$. 
\end{proof}

\begin{remark}\label{kahler}
Assume that we have a (strictly positive) Hermitian metric on $S^*$ with metric form $\omega$.
Then $\omega=dd^c\log|s|^2$ for any non-vanishing local section of $S$.  
Now $|\psi|$ has a global meaning, $|s\psi|^2=|s|^2|\psi|^2$, and 
$dd^c\log|\psi|^2_\circ=dd^c\log|s\psi|^2 
=dd^c\log|\psi|^2+\omega.$
Thus
\begin{equation}\label{mortal}
dd^c\log|\psi|^2_\circ=dd^c u+\omega, \quad u=\log|\psi|^2.
\end{equation}
If we assume that $E$ is a trivial bundle with a trivial metric, then $dd^c\log|\psi|^2_\circ\ge 0$ and 
by \eqref{mortal} therefore  $u$ is quasi-psh with respect to $\omega$.
The currents $[dd^c u+\omega]^k$ and their residues $\1_Z[dd^c u+\omega]^k$
were introduced for arbitrary $k$ and studied in \cite{ABW}, and further in, e.g., \cite{Bl}.  
Here $u$ can be any $\omega$-psh function with analytic singularities.  Analogues
for other classes of $\omega$-psh functions are studied in \cite{AWW}. 
 \end{remark}

\subsection{Regular embeddings}\label{regem}
Let $g$ be a section of $F\to X$ and let $\J$ be the ideal sheaf generated by $g$. 
We have a non-reduced subspace $\iota\colon Z_\J\to X$ with structure sheaf
$\Ok_{Z_\J}=X_\Ok/\J$.
If the zero set of $\J$ has codimension $\kappa$, and in addition 
$\J$ is locally 
generated by $\kappa$ holomorphic functions, then one says that $\iota$ is a regular embedding. In this case,
$Z_\J$ has a well-defined normal bundle $\No$ over $Z$
and $g$ defines a canonical embedding of $\No$ in $F$.  
If we equip $\No$ with the induced metric, then we have a
well-defined Segre form $s(\No)$ over $Z$. 
Let $[Z_\J]$ denote the Lelong current of the fundamental class of $Z_\J$. 
Then $[Z_\J]=\sum_j a_j [Z_j]$, where $Z_j$ are the irreducible components
of $Z$ and $a_j$ are positive integers.
We  have the generalization
\begin{equation}\label{packad}
M^g=s(\No)\w [Z_\J],
\end{equation}
of \eqref{mg}, see \cite[Proposition~1.5]{aeswy1}. 

If $\psi$ is a section of $F\otimes  S^*$ as in Section~\ref{twistma} that defines a regular embedding, then we have an embedding
$\No\otimes S\to F$ obtained from the embedding $\No\to F\otimes S^*$ induced by $\psi$. Now
\begin{equation}\label{twistreg}
\mathring M^\psi=s(\No\otimes S)\w [Z_\J].
\end{equation}
In fact, if $s$ is a local non-vanishing section of $S$, then  
by  \eqref{packad},   
$\mathring M^\psi=M^{s\psi}=s(\No\otimes S)\w [Z_\J]$, and so \eqref{twistreg} follows.

\subsection{Rank of a holomorphic mapping}\label{kattrank}
Assume that  $W$ is  irreducible and $f\colon W\to Z$ is any holomorphic  mapping.
Then the rank of $f$ at $y$, 
$\dim W-\dim f^{-1}(f(y))$, is lower semi-continuous on $W$ and its maximum, $\rank f$, is
attained on a dense open subset of $W_{reg}$, see, e.g., \cite[II, Section 8.1]{Dem}.
We have, \cite[Corollary II.8.6]{Dem},

\begin{prop}\label{katt}
If $f$ is surjective, then $\rank f=\dim Z$.
\end{prop}

\vspace{.5cm}

\section{Extensions of subbundles}\label{extsub}

If $g\colon E\to F$ is a morphism on $X$,   then outside an analytic subvariety $Z_0$ of positive codimension
$g$ has constant and optimal rank, and thus    
$\Im g$ and $\Ker g$ are subbundles of $F$ and $E$, respectively, in $X\setminus Z_0$ (recall that $X$ is always
assumed to be connected).

\begin{lma}\label{utvid}
Let $S=\oplus_{j=1}^rS_j$ be a direct sum of line bundles $S_j\to X$.
If $g\colon E\to S$ is a morphism that has optimal rank  in $X\setminus Z_0$, 
then there is a modification $\pi\colon \widetilde X\to X$
such that $\Ker \pi^* g$ has an extension across $\pi^{-1} Z_0$ as a holomorphic subbundle of $\pi^*E$.
\end{lma}

Since $\Im g^*=(\Ker g)^\perp$ the
lemma can be rephrased:  If $g^*\colon S^*\to E^*$ has optimal rank in $X\setminus Z_0$, then
the pullback to $\widetilde X\setminus\pi^{-1}Z_0$ of the subbundle $\Im g^*$ has an extension to $\widetilde X$.

\begin{proof}
Let us assume that the optimal rank is $\rho$. 
Let $g_j\colon E\to S_j$, $j=1,2,\ldots$ and let
 $i_1$ be the first index such that  
$g_{i_1}$ is not identically $0$. 
Let $\pi_1\colon X_1\to X$ be a modification such that $\pi_1^* g_{i_1}=g_1^0 g_1'$, where $g_1^0$ is a section
of a line bundle $L_1\to X_1$ and $g_1'$ is a non-vanishing section of $\pi_1^* E\otimes L_1^*$. Then
$N_1:=\Ker g_1'$ is a subbundle of $\pi_1^*E$ of codimension $1$ over $X_1$.
Let now $i_2>i_1$ be the first index so that 
$\pi_1^*g_{i_2}|_{N_1}\colon N_1\to S_{i_2}$ does not vanishing identically.
Then there is a a modification $\pi_2\colon X_2\to X_1$ such that $\pi_2^*\pi_1^*g_{i_2}= g_2^0 g_2'$, where $g_2'$ is non-vanishing. Hence $N_2:=\Ker g_2'$ is a subbundle of $\pi^*_2 \pi^*_1E$ of codimension $2$ over $X_2$.  Proceeding in this way 
we end up with a subbundle $N_\rho $ of $\pi^*E$ over $\tilde X=X_\rho$,
where 
$\pi=\pi_1\circ \cdots \pi_\rho\colon \widetilde X\to X$.  In the Zariski-open subset of $\widetilde X$
where $\pi$ is a biholomorphism, $N_\rho=\cap_j \Ker \pi^* g_j=\Ker \pi^*g$ and hence $N_\rho$ is 
the desired extension to $\widetilde X$. 
\end{proof}

\begin{prop}
Assume that $E,F$ are Hermitian bundles and $g\colon E\to F$  has optimal rank in $X\setminus Z_0$.  
Then the natural extensions  from $X\setminus Z_0$ to $X$ of 
$s(E/\Ker g)$ and $s(\Im g)$ as well as of $c(E/\Ker g)$ and $c(\Im g)$ are locally integrable in $X$.
\end{prop}

If $\pi\colon \widetilde X\to X$ is a modification, then it is generically one-to-one and hence $\pi_*1=1$.
It follows that $\pi_*\pi^* a= a$ if $a$ is a smooth form on $X$.

\begin{proof}
In a \nbh $U$ of any given point $x\in X$ both $E$ and $F$ are trivial and by Lemma~\ref{utvid} there is
a modification $\pi\colon\widetilde U\to U$ such that $\Im \pi^*g$ and $\Ker \pi^*g$ have extensions
from $\widetilde U\setminus \pi^{-1} Z_0$ to $\widetilde U$.   Since these extensions are subbundles
of $p^*F$ and $p^*E$, respectively, they inherit Hermitian metrics. 
In $\widetilde U\setminus \pi^{-1} Z_0$ we have $\pi^* s(E/\Ker g)=s(\pi^* E/\Ker \pi^* g)$, and thus 
\begin{equation}\label{maskin}
s(E/\Ker g)=\pi_* s(\pi^* E/\Ker \pi^* g)
\end{equation}
in $U\setminus Z_0$. 
Since the Hermitian bundle $\pi^* E/\Ker \pi^* g$ has an extension to $\widetilde U$,  
$s(\pi^* E/\Ker \pi^* g)$ has a smooth extension to $\widetilde U$, in particular it is
locally integrable, and hence
$\pi_*s(\pi^* E/\Ker \pi^* g)$ is locally integrable in $U$. 
In view of \eqref{maskin} it coincides with $s(E/\Ker g)$ in $U\setminus Z_0$ and since $Z_0$ is a set
of measure zero, thus $\1_{U\setminus Z_0}s(E/\Ker g)$ is locally integrable. 
The other statements are proved in the same way.
\end{proof}

\begin{lma}\label{utvidgning} 
If  $X$ is  compact and projective and $g\colon E\to F$ is any morphism,
then there is a modification $\pi\colon \widetilde X\to X$ such that
both $\Ker \pi^*g$ and $\Im \pi^* g$ have bundle extensions to $\widetilde X$.
\end{lma}

\begin{proof}
Let $L\to X$ be an ample line. Since  $\F= {\mathcal Ker}  (\Ok(E)\stackrel{g}{\to} \Ok(F))$ is a coherent sheaf,
and $X$ is compact, $\F\otimes L^\kappa$ is generated by a finite number of
global sections if $\kappa$ is large enough, see, e.g., \cite[Theorem~1.2.6]{Laz}.  
 If $S_j=L^{-\kappa}$ and  $S=\oplus_1^r S_j$,
 we therefore have a morphism $h$ so that  
$
\Ok(S)\stackrel{h}{\to} \Ok(E)\stackrel{g}{\to} \Ok(F)
$
is an exact sequence of sheaves. It follows that 
$
S\stackrel{h}{\to} E\stackrel{g}{\to} F
$
is a generically exact complex of vector bundles. By Lemma~\ref{utvid} there is a modification such that
$\Im \pi^* h$ has a bundle extension to $\widetilde X$. Since it coincides generically with $\Ker \pi^* g$,
therefore  $\Ker \pi^* g$ has the same extension to $\widetilde X$. 

In the same way we can find a similar bundle $S^*$ and a homomorfism $f$ such that
$S^*\stackrel{f}{\to} F^*\stackrel{g}{\to} E^*$
is generically exact. Hence 
$
E\stackrel{g}{\to} F \stackrel{f}{\to} S
$
is generically exact and it follows from Lemma~\ref{utvid}
that there is a further modification such that $\Ker \pi^* f$ and hence $\Im \pi^*g$  
have bundle extensions to $\widetilde X$.
\end{proof}

\begin{remark}
Following the proof of Lemma~\ref{utvid} we can produce a local holomorphic frame for
the extension of $\Ker g$.  To simplify notation we suppress all $\pi_j$. We can assume
that all $S_j$ are trivial so that $g_j$ are just sections of $E^*$. Moreover,
we can assume that $r=\rho$, since otherwise we delete 'unnecessary' $g_j^*$ from the beginning.  
Now $g_1=g_1^0g_1'$, where $g_1'$ is non-vanishing and hence defines a subbundle of
$E^*$ of rank $1$, or equivalently a subbundle $N_1$ of $E$ of codimension $1$. 
Locally we can find a section $e_1^*$ of $E^*$ that is parallell with $g_1'$ so that
$g_1=\alpha_{11}e_1^*$. By assumption the restriction of $g_2$ to $N_1$ does not
vanish identically. Thus after a further modification $g_2=g_2^0 g_2'$ where $g_2'$ is non-vanishing
on $N_1$. We can choose a local section $e_2^*$ of $E^*$ such that its image in $N_1$ is
parallell with $g_2'$. It follows that $g_2=\alpha_{21}e_1^*+\alpha_{22}e_2^*$.
Proceeding in this way we get linearly independent sections 
$e_1^*, \ldots, e_r^*$ of $E^*$ such that $N$ is subbundle of $E$ that annihilates all of them. 
Moreover, for $\ell=1,\ldots,r$,
$$
g_\ell=\alpha_{\ell1}e_1^*+\cdots +\alpha_{\ell\ell}e_\ell^*,
$$
where $\alpha_{\ell\ell}$ does not vanish identically. Notice that 
$
\det g=g_1\w\ldots \w g_r=\alpha_{11}\cdots\alpha_{rr} e_1^*\w\ldots\w e_r^*.
$
If we extend $e_j^*$ to a local frame
$e_1^*, \ldots, e_m^*$ for $E^*$ and let $e_1,  \ldots, e_m$ be he dual frame for $E$, then
$N$ is spanned by $e_{r+1}, \ldots, e_n$.
 \end{remark}




\vspace{.5cm}

\section{Definition of $M^g$ and the main result Theorem~\ref{thmett}}\label{defsec}

First assume that $E$ is a line bundle so that $g$ is a section of
$F\otimes E^*$.  
We define 
\begin{equation}\label{ragnar}
M^g=s(E)\w\sum_{\ell=0}^\infty \1_Z [dd^c\log|g|_\circ^2]^\ell,
\end{equation}
where $|g|_\circ$ means that we suppress $E^*$ so that locally
$dd^c\log|g|_\circ^2=dd^c\log|ag|^2$ for any non-vanishing section $a$ of $E$,
cf.~Section~\ref{twistma}.

From now on we assume that $r=\rank E\ge 2$.
Let $\Pk(E)$ be the projectivization of $E$, let $p\colon \Pk(E) \to X$ be the
natural projection, and let $L\subset p^* E$ be the tautological bundle,
cf.~Section~\ref{prel2}. 
Notice that a local section $\sigma$ of $L$ has the form 
\begin{equation}\label{greta}
\sigma=s(x,\alpha)\alpha
\end{equation}
at $(x,[\alpha])$, $\alpha\in E_x$, where $s(x,\alpha)$ is a holomorphic function on
$E\setminus\{\bf 0\}$, $\bf 0$ denoting the zero section, that is 
$-1$-homogeneous  in $\alpha\in E_x\setminus\{0\}$. 
By \eqref{greta} we can identify sections $\sigma$ of $L$ with  such $s(x,\alpha)$, and thus consider $\alpha$
as a section of $p^*E\otimes L^*$. 
Therefore, cf.~Section~\ref{twistma},
$dd^c\log|\alpha|^2_\circ:=dd^c\log|s\alpha|^2$ is a global form on $\Pk(E)$, and in fact
equal to $dd^c\log|\sigma|^2=-c_1(L)$, cf.~\eqref{plbas}.  Thus $c(L)=1-dd^c\log|\alpha|^2_\circ$  
so that  
\begin{equation}\label{defseg}
s(L)=\sum_{\ell=0}^\infty \omega_\alpha^\ell,   \quad \omega_\alpha=dd^c\log|\alpha|_\circ^2.
\end{equation}

Since $g$ induces a morphism $p^*E\to p^*F$,  in particular it defines a morphism
$L\to p^*F$.  A local section of $L$, represented by the  $-1$-homogeneous function $s(x,\alpha)$ as above,
is mapped to the well-defined section
$s(x,\alpha)g(x)\alpha$ of $p^*F$.  Thus
\begin{equation}\label{noxdef}
G(x,\alpha):=g(x)\alpha
\end{equation}
is a holomorphic section of $p^*F\otimes L^*\to \Pk(E)$.

Let $Z'$ be the zero set of $G$ on $\Pk(E)$. 
 As  before, let $Z$ be the set where $g$ is not injective and let $Z_0$ be the set where $g$ does not have optimal rank.
If $g$ is generically injective, then $Z=Z_0$ and
$Z'\subset p^{-1}Z_0$.  If $g$ is not generically injective, then $Z=X$ and $p(Z')=X$. 
If $N=\Ker g$, then $\Pk(N)$ is a submanifold of $\Pk(E)$ in $ p^{-1}(X\setminus Z_0)$ and 
$$
Z'\cap p^{-1}(X\setminus Z_0)=\Pk(N)\cap  p^{-1}(X\setminus Z_0).
$$

\smallskip

Letting $|g\alpha|_\circ=|G|_\circ=|sG|=|sg\alpha|$, where $s$ is a local non-vanishing section of $L$, we have,
following Section~\ref{twistma},
the generalized Monge-Amp\`ere powers  
$
[dd^c \log |g\alpha|_\circ^2]^\ell, \quad \ell=0,1,2,\ldots
$
and their residues
$
\mathring M^{g\alpha}_\ell= \1_{Z'}[dd^c\log|g\alpha|^2_\circ]^\ell,  \quad \ell=0,1,2,\ldots.
$
Locally on $\Pk(E)$ thus
\begin{equation}\label{ringa}
\mathring M^{g\alpha}=M^{sg\alpha}.
\end{equation}

\begin{df}\label{mdef}
We define 
$M^g=p_*\big(s(L)\w \mathring M^{g\alpha}\big)$. 
\end{df}
Thus  
$M^g=M_0^g+M_1^g+\cdots + M_n^g$, where $M_k^g=p_*(s(L)\w \mathring M^{g\alpha})_{k+r-1}$ are
closed $(k,k)$-currents with support on $Z$. 
Notice that $\mathring M^{g\alpha}$ and $s(L)$  only depend on the metrics on $F$ and $E$, respectively.

\begin{ex}\label{trivial}
Assume that $E$ and $F$ are trivial and have  trivial metrics. We can assume that $E=\C^r_\alpha\times X$,
$F= X\times \C^m$, with the Euclidian metric on $\C^r_\alpha$ and $\C^m$. Then 
$\Pk(E)= X\times \Pk(\C^r_\alpha) $ and $\omega_\alpha=dd^c\log|\alpha|_\circ^2$ is the usual 
Fubini-Study metric form on $\Pk(\C^r_\alpha)$; in particular it is a positive form.
Thus $s(L)$, cf.~\eqref{defseg}, is 
independent of $x$.  Moreover, locally, for any non-vanishing holomorphic $-1$-homogeneous $s$,
e.g., $s=1/\alpha_j$ in the open set  $\alpha_j\neq 0$, 
$$
[dd^c\log|g\alpha|^2_\circ]^\ell=[dd^c\log|sg\alpha|^2]^\ell
$$
is a positive current. 
Since $\omega_\alpha$ is a positive $(1,1)$-form therefore, cf.~\eqref{defseg},  $s(L)\w \mathring M^{g\alpha}$ is a positive current on $\Pk(E)$, and thus
$M^g$ is a positive current on $X$.
\end{ex}

\begin{df}\label{brutus}
We say that the morphisms $g\colon E\to F$ and $g'\colon E\to F'$  
are {\it comparable} if locally in $X$ 
$$
|g(x)\alpha| \sim |g'(x)\alpha|, \quad  \alpha\in E_x.
$$
 \end{df}

In case $r=1$,  comparability means that the entries in $g$ and $g'$, respectively,  generate
ideal sheaves with the same integral closure,

\begin{thm}\label{thmett}
Let $E$ and $F$ be Hermitian vector bundles over $X$ and
$g\colon E\to F$ a holomorphic morphism.
The following holds:

\smallskip
\noindent (o) 
The currents $M^g_k$  are generalized cycles, smooth in the Zariski-open set $X\setminus Z_0$ where $g$ has optimal rank,
and  positive  on $X$ if  $E$ and $F$ have trivial metrics.

\smallskip
\noindent (i) 
The natural extension $\1_{X\setminus Z}s(\Im g)$ to $X$ of $s(\Im g)$ is locally integrable and closed,
and there is a current $W^g$ with singularities  of logarithmic type along $Z_0$ such that  
\begin{equation}\label{main11}
dd^c W^g= M^g+  \1_{X\setminus Z}s(\Im g)  -s(E).
\end{equation}

\smallskip
\noindent (ii) 
If $i\colon X'\to X$ is an open subset, then
$M^{i^*g}$ is the restriction of $M^g$ to $X'$. 

\bigskip
\noindent (iii) 
If $\pi\colon \widetilde X\to X$ is a modification, 
then  
$
\pi_* M^{\pi^*g}= M^g.
$

\smallskip
\noindent (iv) If $i \colon F\to F'$ is a subbundle with the metric inherited from $F'$,  then
\begin{equation}\label{sommar62}
M^{i\circ g}=M^g.
\end{equation}

\smallskip
\noindent (v)
If  $g'\colon E'\to F'$ is pointwise injective, then 
\begin{equation}\label{plus}
M^{g\oplus g'}=s(\Im g') \w M^{g}.
\end{equation}

\smallskip
\noindent (vi)
The  multiplicities $\mult_x M_k^g$ are non-negative integers.

 \smallskip
\noindent (vii)
If $g$ and $g'$ are comparable, then $\mult_x M^g_k=\mult_x M^{g'}_k$ for each $k$ and each point $x$.

\smallskip
\noindent (viii)
For each $k$
we have a unique decomposition
\begin{equation}\label{king}
M^g_k=\sum_j \beta_j^k[Z_j^{k}]+N^g_k=:S^g_k+N^g_k,
\end{equation}
where $Z_j^{k}$ are irreducible subvarieties of codimension $k$, $\beta_k^j$ are positive integers,
and $N^g_k$ is a closed $(k,k)$-current with support on $Z$ whose multiplicities vanish
outside a variety of codimension $\ge k+1$.
Moreover, $\cup_{jk} Z_j^k=Z$.
\end{thm}

Notice that $Z$ has positive codimension if and only if $g$ is generically injective. 
If $g$ is not generically injective, thus
$
dd^c W^g= M^g -s(E),
$
and  $M^g$ is smooth in the open set $X\setminus Z_0$ where $g$ has optimal rank,  see  part (o)
and Proposition~\ref{guppi}. 

\smallskip\noindent
By the dimension principle for normal currents, $M^g_k=0$ if $k<\codim Z$.  

\smallskip\noindent
If $g$ is generically injective, then $E$ and $\Im g$ are isomorphic in $X\setminus Z$ so that
$s(E)$ and $s(\Im g)$ define the same Bott-Chern (cohomology) class there. Equality \eqref{main11} is 
an extension across $Z$. If $g$ is not generically injective, then $M^g$ is a representative
of the Bott-Chern cohomology class $\hat s(E)$.

 \smallskip\noindent
A variant of (iii) holds for a general proper mapping $h$,  see Proposition~\ref{kakadua}.
Regarding $(v)$, notice that $\Im (g\oplus g')= \Im g \oplus \Im g'$ in $X\setminus Z$ 
and $s(\Im g\oplus \Im g')=s(\Im g)\w s(\Im g')$, and thus  
\eqref{plus} is consistent with \eqref{main11}.

\smallskip
\noindent 
Parts (vii) and (viii) of Theorem~\ref{thmett} imply that 
if $g$ and $g'$ are comparable, then $M^g_k$ and $M^{g'}_k$ have 
the same fixed part.

\vspace{.5cm}

\section{Proofs of Theorem~\ref{thmett} and  Proposition~\ref{avig}}\label{proofsection} 
First we need some preparations.  We keep the notation from Section~\ref{defsec}. In particular, recall
that $N=\Ker g$ over $X\setminus Z_0$. 

\begin{lma}\label{genlemma}  Assume that $g$ is not generically injective. Then the section
$G$ generates the ideal defining $\Pk(N)$ in $\Pk(E)\setminus p^{-1}Z_0$. 
\end{lma}

\begin{proof}
Locally in $X\setminus Z_0$ we can choose a trivialization $E=\U\times \C_\alpha^r$ such that
$N=\{\alpha_1=\cdots =\alpha_\rho=0\}$. Let $\alpha=(\alpha', \alpha'')=(\alpha_1, \ldots \alpha_\rho, \alpha'')$. 
Then $\alpha'\mapsto g(\alpha',\alpha'')=g(\alpha',0)=g'\alpha'$ is injective, and hence there is $h$ such that
$hg'\alpha'=\alpha'$. Thus $\la g'\alpha'\ra=\la \alpha'\ra$ so that $g\alpha=g'\alpha'$ generates $N$.
Now the lemma follows. 
\end{proof}

By the lemma $G$ defines a regular embedding in $\Pk(E)\setminus p^{-1}Z_0$ and thus, cf.~Section~\ref{regem},  it induces
an embedding  $\iota \colon \mathcal N_{\Pk(N)} \to p^*F\otimes L^*$ and hence a mapping
$\iota \colon \mathcal N_{\Pk(N)} \to p^*\Im g\otimes L^*$
on $\Pk(N)$.
For dimension reasons $\iota$ and hence the induced mapping 
%
%
\begin{equation}\label{kalle}
 \mathcal N_{\Pk(N)}\otimes L\simeq p^*\Im g 
 \end{equation}
 must be isomorphisms on  $\Pk(N)\setminus p^{-1}Z_0$.


\begin{remark}
One can establish the isomorphism \eqref{kalle} in a more direct way.  Recalling that $T\Pk(E)=p^*E/[\alpha]$ and similarly 
$T\Pk(N)=p^*N/[\alpha]$ we have 
$$
 \mathcal N_{\Pk(N)}=T\Pk(E)/T\Pk(N) =p^*E/[\alpha]\big/ p^*N/[\alpha].
$$
Since $(x,\gamma)\mapsto g(x)\gamma\in F\otimes L^*$ is injective on $p^*E/[\alpha]\big/ p^*N/[\alpha]$, 
the isomorphism follows. 
\end{remark}

We conclude that 
$$
s(\No_{\Pk(N)}\otimes L)=s(p^* \Im g)=p^*s(\Im g)
$$
on $\Pk(N)\setminus p^{-1} Z_0$. 
From \eqref{twistreg} 
we have the representation 
%
%
 %
\begin{equation}\label{ringsegre}
\mathring M^{G}=p^*s(\Im g)\w [\Pk(N)].
\end{equation}
Let $p'\colon \Pk(N)\to X$ be the natural projection. Then by \eqref{ringsegre},
\begin{multline*}
M^g=p_*\big(s(L)\w  \mathring M^G\big)=p_*(s(L)\w p^*s(\Im g)\w [\Pk(N)]\big) \\ =s(\Im g)\w p_*\big(s(L)\w[\Pk(N)]\big)=
s(\Im g)\w p'_* s(L)=s(\Im g) \w s(N)
\end{multline*}
on $X\setminus Z_0$.
The last equality holds since 
the restriction of $s(L)$ to $\Pk(N)$ is equal to $s(L')$ where
$L'$ is the tautological line bundle on $\Pk(N)$ with the metric inherited from $E$. 
%
\begin{prop}\label{guppi} 
Assume that $g$ not generically injective. In $X\setminus Z_0$ we have $M^g=s(\Im g)\w s(N)$.
\end{prop}


We now turn our attention to regularizations of $M^g$.  
To begin with we apply Proposition~\ref{circvprop}
to $\mathring M^{g\alpha}$ with $V=Z'$. 
Notice that $|g\alpha|^2/|\alpha|^2$ is a global function on $\Pk(E)$ with
zero set $Z'$ so we can take $\chi_\e=\chi(|g\alpha|^2/\e |\alpha|^2)$.  Also notice that
$g\alpha$ is a section of $p^* F\otimes L^*$ and that
$\alpha$ is a non-vanishing section $\alpha$ of $p^*E\otimes L^*$.  By Proposition~\ref{circvprop}
the smooth forms  
$$
1-\chi_\e +\dbar\chi_\e\w \frac{\partial\log(|g\alpha|/|\alpha|)^2}{2\pi i}\w \sum_{\ell=0}^\infty\la dd^c\log|g\alpha|_\circ^2\ra^\ell
$$
on $\Pk(E)$ tend to $\mathring M^{g\alpha}$.  Since $p\colon \Pk(E)\to X$ is a submersion
we  get

\begin{prop}\label{mreg1}
With the notation above the forms\footnote{The term $1-\chi_\e$ can be omitted unless $g\equiv0$.}
\begin{equation}\label{fasan}
M^{g,\e} := p_*\Big(s(L)\w\big(1-\chi_\e +\dbar\chi_\e\w \frac{\partial\log(|g\alpha| /|\alpha|)^2}{2\pi i}\w \sum_{\ell=0}^\infty\la dd^c\log|g\alpha|_\circ^2\ra^\ell\big)\Big)
\end{equation}
are smooth on $X$ and tend to $M^g$ when $\e\to 0$.
\end{prop}


Let us now assume that $g\colon E\to F$ is generically injective. That is,
$Z=Z_0$ and $p(Z')=Z_0$. Then $p^{-1}Z$ has positive codimension in $\Pk(E)$ so we can take
$V=p^{-1}Z$ in Proposition~\ref{circvprop}. 
Moreover,  $Z$  is  the zero set of the global section $\varphi=\det g$ of 
$\Lambda^r E^*\otimes\Lambda^r F$, where $r=\dim E$. 
In local frames for $E$ and $F$ it is the tuple of all  $r\times r$ minors of the associated matrix. 
Let $\chi_\e=\chi(|\det g|^2/\e)$, and for simplicity let us write $\chi_\e$ also for $p^*\chi_\e$.
  
\begin{prop}\label{mreg2}
Assume that $g$ is generically injective and $\chi_\e=\chi(|\det g|^2/\e)$. For each $\e>0$ 
\begin{equation}\label{premium}
M^{g,\e} := p_*\Big(s(L)\w\dbar\chi_\e\w \frac{\partial\log\big(|g\alpha|/ |\alpha|)^2}{2\pi i}\w \sum_{\ell=0}^\infty\la dd^c\log|g\alpha|_\circ^2\ra^\ell\Big)
\end{equation}
is a  smooth forms on $X$ that  vanishes in a \nbh of $Z=Z_0$,  and the sequence tends to $M^g$ when $\e\to 0$.
\end{prop}

\begin{remark}
 It follows from Proposition~\ref{circvprop} that the currents in \eqref{premium} tend to $\1_{Z_0}M^g$, 
 provided that $g$ is not identically $0$, 
 if we choose $\chi_\e$  that converges to $\1_{X\setminus Z_0}$. If the optimal rank of $g$ is $\rho$ we can take
 $\varphi$ as the $\rho$-determinant of $g$.  Still each $M^{g,\e}$ vanishes in a \nbh of $Z_0$ but it is
 not smooth in general. 
 \end{remark}

\subsection{Proof of (o) and (i) of Theorem~\ref{thmett}}
By Proposition~\ref{guppi} $M^g$ is smooth in $X\setminus Z_0$.
Lemma~\ref{ringlemma}   claims that $\mathring M^{g\alpha}$ is an element in $\GZ(\Pk(E))$ and, cf.~Section~\ref{gztrams},
therefore $s(L)\w \mathring M^{g\alpha}$ is in $\GZ(\Pk(E))$. Since $p\colon\Pk(E)\to X$ is
proper, cf.~Definition~\ref{mdef}, $M^g_k=p_*\big((s(L)\w \mathring M^{g\alpha})_{k+r-1}\big)$ is in $\GZ_{n-k}(X)$ for $k=0,1,2,\ldots$.
If the metrics on $E$ and $F$ are trivial, then $M^g$ is positive, cf.~Example~\ref{trivial}.
Thus (o) holds. 

\smallskip
If $g\equiv 0$, then $Z=X$ and $M^g=s(E)$, and so (i) is trivial. 
Let us therefore assume that $g$ is not identically $0$.  
We first consider the case when $E$ is a line bundle so that
%
$g$ is a section of $F\otimes E^*$. 
Let $a$ be a local non-vanishing section of $E$. In $X\setminus Z$ then 
$ga$ is a non-vanishing section of the line bundle $\Im g$. Therefore
the locally integrable currents, cf.~Section~\ref{defsec}, 
$$
\la dd^c\log|ga|^2\ra, \quad \sum_{\ell=0}^\infty \la dd^c\log|ga|^2\ra^\ell,
$$
are equal to $s_1(\Im g)$ and $s(\Im g)$, respectively in $X\setminus Z$.
Moreover, cf.~Section~\ref{defsec}, 
$$
w^g:=\log(|ga|^2/|a|^2)  s(E)\w \1_{X\setminus Z}s(\Im g) 
$$
is locally integrable in $X$. 
In $X\setminus Z$ we have 
$
dd^c \log(|ga|^2/|a|^2)=s_1(\Im g)-s_1(E).
$
Thus 
\begin{multline}\label{anita1}
\1_{X\setminus Z}dd^c w^g= 
\1_{X\setminus Z} (s_1(\Im g)-s_1(E))\frac{1}{1-s_1(E)}\frac{1}{1-s_1(\Im g)}=\\
\1_{X\setminus Z} \Big( \frac{1}{1-s_1(\Im g)}-\frac{1}{1-s_1(E)}\Big)=
\1_{X\setminus Z} s(\Im g)-s(E)
\end{multline}
whereas 
\begin{multline}\label{anita2}
\1_Z  dd^c w^g=\\
s(E)\w \sum_{\ell=0}^\infty\1_Z dd^c(\log(|ga|^2  \la dd^c\log|ga|^2\ra^\ell\big)= 
s(E)\w \1_Z\sum_{\ell=1}^\infty [dd^c\log|ga|^2]^\ell =M^g,
\end{multline}
cf.~\eqref{ragnar},
since $s(E)$ is smooth and closed and 
$$
\1_Z dd^c\big(\log|a|^2  \la dd^c\log|ga|^2\ra^\ell\big)=s_1(E) \1_Z \la dd^c\log|ga|^2\ra^\ell=0.
$$
Part (i) of Theorem~\ref{thmett}  now follows from \eqref{anita1} and \eqref{anita2} in case $\rank E=1$.


\vspace{.3cm}

Let us now assume that $r=\rank E\ge 2$.  
We keep the notation from Section~\ref{defsec}. Let $p'\colon \Pk(F)\to X$  and 
let $L'$ be the tautological line bundle in $(p')^*F\to \Pk(F)$.  
Notice that $g$ induces a holomorphic mapping $\tilde g\colon (\Pk(E)\setminus Z')\to \Pk(F)$ and so $\tilde g^*L'$ is a
well-defined line bundle over $\Pk(E)\setminus Z'$.  Moreover $p=p'\circ \tilde g$.
From now on we write $g$ rather that $\tilde g$ for notational simplicity.  If $s'$ is a section of $L'$ then
$s=g^* s'$ is a section of $L$.  Therefore,  
 since $g(x,[\alpha])=(x,[g(x)\alpha])$ 
 (if $\beta$ denotes elements in $F$)
 $$
g^*s_1(L')=g^* dd^c\log|\beta|^2_\circ=g^* dd^c\log|\beta s'|^2=dd^c\log|g\alpha s|^2=dd^c\log|g\alpha|_\circ^2.
$$
In view of \eqref{defseg} it is natural to introduce the locally integrable form  
\begin{equation}\label{sedan5}
g^*s(L'):=\sum_{\ell=0}^\infty \1_{\Pk(E)\setminus Z'}[dd^c\log|g\alpha|_\circ^2]^\ell=
\sum_{\ell=0}^\infty \la dd^c\log|g\alpha|_\circ^2\ra^\ell.
\end{equation}
on $\Pk(E)$. 
 
\begin{lma}\label{lmakork}
We have that  
\begin{equation}\label{sedan3}
p_* g^*s( L')=\1_{X\setminus Z} s(\Im g).
\end{equation}
\end{lma}


\begin{proof}
First assume that $g$ is generically injective. 
 Then $p^{-1}Z$ has positive codimension in $\Pk(E)$ and therefore
$\1_{p^{-1}Z}g^*s(L')=0$. Thus it is enough to prove \eqref{sedan3} in $X\setminus Z$.
There $g\colon E\to \Im g$ is an isomorphism and hence $g\colon \Pk(E)\to \Pk(\Im g)$
is a biholomorphism  and so $g^*=g^{-1}_*$.
Moreover, the restriction of $L'\to \Pk(F)$ to $\Pk(\Im g)$ is the tautological line bundle over $\Pk(\Im g)$;
let us denote this restriction by $L'$. 
Noticing that $p'=p g^{-1}$ we get
$$
p_*g^*s(L')=p_*g^{-1}_* s(L')=(pg^{-1})_* s(L')= (p')_*s(L')=s(\Im g).
$$


%

We now assume that the generic rank of $g$ is $<\rank E$. Then $Z=X$ and so the right hand side of
\eqref{sedan3} vanishes on $X$.  We must ensure that the
left hand side vanishes as well. 
Since $g^*s(L')$ is locally integrable on $\Pk(E)$, $p_* g^*s(L')$ is locally integrable on $X$
and thus it is enough to see that it vanishes
on $X\setminus Z_0$. 
There $\Ker g$ is a subbundle of $E$ of positive dimension.  Let us choose a local frame
$e_1,\ldots, e_{r-1},e_r$ for $E$ so that $e_r$ belongs to $\Ker g$. Then
$E=X\times \C^r_\alpha$, where 
$\alpha=\alpha_1 e_1+\cdots +\alpha_{r-1}e_{r-1}+ \alpha_r e_r$. 
Clearly $g \alpha=g(\alpha_1 e_1+\cdots +\alpha_{r-1}e_{r-1})$. 
In a \nbh of a point on $\Pk(E)$ where, say,  $\alpha_{r-1}\neq 0$, and $g\alpha\neq 0$, we have 
$$
dd^c\log|g\alpha|^2=dd^c\log|g((\alpha_1/\alpha_{r-1}) e_1+\cdots +(\alpha_{r-2}/\alpha_{r-1})e_{r-2}+e_{r-1})|^2.
$$
Locally on $\Pk(E)$, $\alpha'_j=\alpha_j/\alpha_{r-1}$, $j\neq  r-1$, together with $x$ form
a  local system of coordinates, and we see that  
$(dd^c\log|g\alpha|^2)^\ell$ has at most bidegree $(r-2,r-2)$ in $\alpha'$.  
Since $p$ is $(x,[\alpha])\mapsto x$ it follows that the left hand side
of \eqref{sedan3} vanishes. 
\end{proof}

Notice that   
$$
w=\log \big(|g\alpha|^2/|\alpha|^2\big)
$$
is a global function on $\Pk(E)$ which has singularities of logarithmic type
along $Z'$.
We claim that 
\begin{equation}\label{sommar2}
dd^c \big( w s(L) g^*s(L')\big)=s(L)\w \mathring M^{g\alpha}+
g^*s(L')  -s(L) .
\end{equation}
Outside $Z'$ the left hand side of \eqref{sommar2} is
\begin{multline*}
dd^c\Big( w \frac{1}{(1-s_1(L))(1-g^*s_1(L'))}\Big)=\\
\frac{g^*s_1(L')-s_1(L)}{(1-s_1(L))(1-g^*s_1(L'))}=\frac{1}{1-g^*s_1(L')}-\frac{1}{1-s_1(L)}=g^*s(L')-s(L).
\end{multline*}
The only contribution at $Z'$ comes from
the residue term which is  
$$
s(L)\w\1_{Z'} dd^c \frac{ \log|g\alpha|_\circ^2}{1-\langle dd^c\log|g\alpha|_\circ^2\rangle}=
s(L)\w\1_{Z'} \sum_{\ell=1}^\infty [dd^c\log|g\alpha|^2_\circ]^\ell=s(L)\w \mathring M^{g\alpha},
$$
 cf.~\eqref{sedan5}.  
For the last equality we have  used that $Z$ has positive codimension so that $\mathring M_0^{g\alpha}=0$.
Thus \eqref{sommar2} holds.

With a modification $\pi\colon Y\to \Pk(E)$ as in the proof of Lemma~\ref{ringlemma} we see that
 $g^*s(L')=\pi_*s(\La)$ and that $\pi^*w$ locally has the form $\log|\psi^0|^2 + smooth$ on $Y$.
 Thus  
$ws(L) g^*s( L')$  has singularities of logarithmic type
along $Z'$ and hence along $p^{-1}Z$. Therefore
\begin{equation}\label{wdef}
W^g:=p_* (ws(L) g^*s(L'))
\end{equation}
 has singularities of logarithmic  type along $Z$. From 
\eqref{sommar2}, \eqref{sedan3} and Definition~\ref{mdef}  we have
$$
dd^c W^g=p_*(s(L)\w\mathring M^{g\alpha})+p_*g^*s(L')-p_*s(L)=M^g+\1_{X\setminus Z} s(\Im g)-s(E).
$$
Summing up we have proved part (i) of Theorem~\ref{thmett}.

\subsection{Proof of (ii), (iii) and (iv)}
Part (ii) is clear since all definitions and arguments we use are local on $X$. 
Part (iii) is precisely Lemma~\ref{ormen}:

\begin{lma}\label{ormen}
Assume that $g\colon E\to F$ is a morphism and $\pi\colon \widetilde X\to X$ is a
modification. Then we have an induced mapping $\pi^*g\colon \pi^* E\to \pi^*F$ on $\widetilde X$ and
$\pi_* M^{\pi^* g}=M^g$.
\end{lma}

\begin{proof} 
Let $\widetilde E=\pi^* E$. There is a natural mapping $\hat\pi\colon \hat \Pk(\widetilde E)\to \Pk(E)$ so that 

\begin{equation}\label{dia2}
\begin{array}[c]{cccc}
\Pk(\widetilde E)&  \stackrel{\hat\pi}{\longrightarrow}  & \Pk(E) \\
\downarrow \scriptstyle{\tilde p} & &  \downarrow \scriptstyle{p}  \\
\widetilde X & \stackrel{\pi }{\longrightarrow}  & X
\end{array}
\end{equation}
commutes, and  similarly for $F$.  
The morphism $g\colon E\to F$ induces a morphism $\pi^*g\colon \widetilde E\to \widetilde F$
such that, for $y\in\widetilde X$ and $\alpha\in E_{\pi(y)}$,
\begin{equation}\label{primus}
 \pi^*g(y)\alpha=g(\pi(y))\alpha, \quad y\in\widetilde X, \quad \alpha\in E_{\pi(y)},
 \end{equation}
and 
 \begin{equation}\label{dia3}
\begin{array}[c]{cccc}
p^* E & \stackrel{g}{\longrightarrow}  & p^* F \\
\downarrow \scriptstyle{\hat\pi^*} & &  \downarrow \scriptstyle{\hat\pi^*}  \\
 \tilde p^*\widetilde{E} &  \stackrel{\pi^*g}{\longrightarrow}  & \tilde p^*\widetilde{F} 
\end{array}
\end{equation}
commutes. 
%
%
%
%
%
%
 If $L \to \Pk(E)$ is the tautological line subbundle of $p^*E$, then
$\widetilde L:=\hat\pi^*L$ is the tautological subbundle of 
$\hat\pi^* p^*E=\tilde p^* \pi^*E$, cf.~\eqref{dia2}. In particular,
\begin{equation}\label{snoken}
s(\widetilde L)=\hat\pi^* s(L).
\end{equation}
%
Let $s$ be a local non-vanishing holomorphic section of $L$ on $\Pk(E)$. If in addition $g(x)\alpha \neq 0$
and  $\tilde s=\hat\pi^* s$, then by \eqref{primus}, 
\begin{equation}\label{flotus}
\hat\pi^* (sg\alpha)=\tilde s \hat\pi^*(g\alpha), \quad  \hat\pi^* (g \alpha)=\pi^*g \alpha.
\end{equation}
%
Since $\hat\pi$ is generically $1-1$ it follows from \cite[Example~5.3]{aeswy1} that 
$\hat\pi_* M^{\hat\pi^*(sg\alpha)}=M^{sg\alpha}$ where $s$ is defined. 
From \eqref{snoken}, \eqref{flotus}  and the definition of $\mathring M$, cf.~Section~\ref{twistma}, we conclude that
\begin{equation}\label{primus2}
\hat\pi_*\mathring M^{\hat\pi^*(g\alpha)}=\mathring M^{g\alpha}.
\end{equation}
By \eqref{dia2}, \eqref{snoken}, \eqref{flotus}, and \eqref{primus2} thus
$$ 
\pi_* M^{\pi^*g}=\pi_* \tilde p_*\big(s(\widetilde L)\w \mathring M^{\pi^* g\alpha}\big)= \\ 
p_*\hat\pi_*\big(\hat\pi^* s(L)\w \mathring M^{\pi^* (g\alpha)}\big)=
p_*\big( s(L)\w \mathring M^{g\alpha}\big)=M^g.
$$ 
Thus the lemma is proved.
\end{proof}

The definitions and arguments are not affected if we 
consider $g$ as a morphism $E\to F'$ rather than $E\to F$. Thus (iv) follows.

\subsection{Proof of (v)}
Assume that $g'\colon E'\to F'$ is pointwise injective on $X$. 
Let $p\colon \Pk(E)\to X$ and $\hat p\colon \Pk(E\oplus E')\to X$ be the natural mappings.
Moreover, let 
$$
j\colon \Pk(E)\to \Pk(E\oplus E'), \quad [\alpha]\mapsto [\alpha,0].
$$
We claim that 
\begin{equation}\label{surra}
\mathring M^{g\alpha \oplus g'\alpha'}=j_*\big(p^*s(\Im g') \w \mathring M^{g\alpha}\big).
\end{equation}
 To see \eqref{surra},  
 assume that $U\subset X$ is an open set where $E=U\times \C^r_\alpha$ and $E'=U\times \C^{r'}_{\alpha'}$.
It is enough to prove \eqref{surra} in each set $\U_i=\hat p^{-1}U\cap \{[\alpha,\alpha'], \ \alpha_i\neq 0\}$. Let  $i=1$. Then 
$[\alpha,\alpha']$ is represented by 
$$
(1, \alpha_2/\alpha_1, \cdots, \alpha_r/\alpha_1, \alpha'_1/\alpha_1, \ldots \alpha'_{r'}/\alpha_1).
$$ 
The image of $j\colon \Pk(E)\to \Pk(E\oplus E')$ 
is cut out by the section $g'(x)\alpha'/\alpha_1$ of $\hat p^*\Im g'$ over $\U_1$.  Since $\Im g'$ has the same rank as
the codimension, the normal bundle of the image of $j$ is precisely $\hat p^*\Im g'$.  
From  \cite[Lemma~5.9]{Kalm} we have that 
$$
\hat p^*c(\Im g')\w M^{g\alpha \oplus g'\alpha'}=j_*M^{g\alpha}.
$$
Now $\hat p^*s(\Im g') \w j_*M^{g\alpha}= 
j_*(j^*\hat p^* s(\Im g')\w M^{g\alpha})
$ 
and 
thus \eqref{surra} holds in $\U_1$ since
$p^*=j^*\hat p^*$
In the same way it holds in any $\U_i$, $i=1,\ldots, r$, and  so 
\eqref{surra} is proved.

Let $\hat L$ be the tautological line bundle in
$\hat p^*(E\oplus E')\to \Pk(E\oplus E')$, and recall that 
$$
s(\hat L)=\sum_{\ell=0}^\infty (dd^c\log(|\alpha|^2+|\alpha'|^2)_\circ)^\ell.
$$
Since the pullback to $\{\alpha'=0\}$ of $dd^c\log(|\alpha|^2+|\alpha'|^2)_\circ$ is $dd^c\log|\alpha|^2_\circ$, 
 \eqref{surra} implies that 
\begin{equation}\label{plurr}
s(\hat L)\w \mathring M^{g\alpha \oplus g'\alpha'}=j_*\big(j^*s(\hat L)\w p^*s(\Im g') \w \mathring M^{g\alpha}\big)=
j_*\big(s(L)\w p^*s(\Im g') \w \mathring M^{g\alpha}\big)
\end{equation}
Since  $p_*=\hat p_*j_* $ we get from \eqref{plurr} that 
$$
M^{g\oplus g'}=\hat p_*\big(s(\hat L)\w \mathring M^{g\alpha \oplus g'\alpha'}\big)=
p_*\big(p^*s(\Im g')\w  s(L)\w  \mathring M^{g\alpha}\big)=
s(\Im g')\w M^g.
$$
Thus (v) is proved.

\subsection{Proof of (vi) and (vii)}
If $g'$ is a morphism such that $|g'\alpha|\sim |g\alpha|$, then by  Lemma~\ref{ringlemma}
$\mathring M^{g\alpha}$ and $\mathring M^{g'\alpha}$ define the same class in $\mathcal B(\Pk(E))$. 
It follows that $M^g$ and $M^{g'}$ define the same class in $\mathcal B(X)$. Therefore 
the multiplicities of $M^g$ and $M^{g'}$ at each point $x\in X$ coincide, and are integers. 
Locally, cf.~Example~\ref{trivial},  we can choose metrics so that $M^g$ is a positive current. We conclude that
the multiplicities are non-negative integers. Thus (vi) and (vii) are proved.

 \subsection{Proof of  (viii)}
 Since $M^g_k$ is in $\GZ_{n-k}(X)$, the decomposition~\eqref{king} follows from 
\eqref{siu}.  
The last statement in (viii) requires an additional argument:
Let $Z_i'$ be the subvarieties of $\Pk(E)$  that appear in the decomposition \eqref{siu} of $M^G_\ell$ for various $\ell$.
It is well-known, and follows from Section~\ref{ma-produkt},  that their union is precisely the zero set $Z'$ of $G$. It is clear that 
$p(Z')=Z$. 
Thus it is enough to prove,  for each $Z_i'$, that $[p(Z'_i)]$ appears in the fixed part in 
 \eqref{king} if $p(Z'_i)$ has codimension $k$ in $X$. 
 It is enough to prove this locally on $X$, so we can assume that the metrics are trivial, keeping in mind that
the fixed part only depends on the class of $M^g_k$ in $\mathcal B(X)$.
If $p|_{Z'_i}$ has generic rank $\rho=n-k$  cf.~Section~\ref{kattrank},
then the generic dimension of the fibers $(p|_{Z'_i})^{-1} x$, $x\in p(Z'_i)$,
is $\nu=\dim Z'_i-\rho$.  If locally $E=X\times \C^r_\alpha$, then $p$ is $([\alpha],x)\mapsto x$ and 
$\omega_\alpha =dd^c\log|\alpha|_\circ$  is strictly positive on each fiber. Therefore  
$p_*( \omega_\alpha^\nu \w[Z'_i])$ has support on $p(Z'_i)$, is non-zero, and has bidegree $(k,k)$.
Hence it is $c [p(Z'_i)]$ for some integer $c\ge 1$.  
It follows that  
$$
M^g_k=p_*((s(L)\w \mathring M^{g\alpha})_{k+r-1})= c[p(Z_i)'] +\cdots
$$
where all terms in $\cdots$ are non-negative, cf.~Example~\ref{trivial}. 
 The proof of Theorem~\ref{thmett} is complete.     

\begin{proof}[Proof of Proposition~\ref{avig}]
Assume that $g\colon E\to F$ and $W^g$  are as in
\eqref{main11}. By \eqref{main11} and \eqref{chern0},
\begin{equation}\label{konkurs}
dd^c (c(F) \w W^g)= c(F)\w M^g +c(F)\w \1_{X\setminus Z}s(\Im g) -  c(F)
\end{equation}
since $E$ is trivial so that $s(E)=1$. By Lemma~\ref{utvid} there  is a modification
$\pi\colon\widetilde X\to X$ such that $\Im \pi^*g$ has an extension to a subbundle $H$ of $\pi^* F$. 
In $\widetilde X$ we thus have the pointwise exact sequence
$$
0\to H \to \pi^*F\to \pi^* F/H\to 0.
$$
By \eqref{chern1} there is a smooth form $v$ such that
$dd^c v=c(\pi^* F)-c(\pi^* F/H)\w c(H)$.
Hence 
$$
dd^c (s(H) \w v)=c(\pi^* F)\w s(H)-c(\pi^* F/H).
$$
Applying $\pi_*$ we see that $\1_{X\setminus Z}c(F/\Im g)$ is locally integrable
and closed, and 
\begin{equation}\label{billdal}
dd^c \pi_*(s(H) \w v)=c(F) \w \1_{X\setminus Z}s(\Im g)- \1_{X\setminus Z}c(F/\Im g).
\end{equation}
From \eqref{konkurs} and \eqref{billdal} we see that \eqref{main123} holds
with $V^g=c(F) \w W^g- \pi_*(s(H) \w v)$.
\end{proof}

 \subsection{A remark}
Here is an alternative way to find regularizations of $M^g$. 
Let us introduce the Hermitian norm on $p^*F\otimes L^*\to\Pk(E)$ so that
$
|G|=|g\alpha|/|\alpha|
$
and consider the current $M^G$, cf.~Remark~\ref{kahler} above. 

\begin{lma}\label{snorkel}
For $k=0,1,2,\ldots$ we have the relations
 \begin{equation}\label{snorig}
\langle dd^c\log|g\alpha|^2_\circ\rangle^{k}=\sum_{j=0}^k {{k}\choose{j}} \langle dd^c \log|G|^2\rangle^j \w \omega_\alpha^{k-j}
\end{equation}
and
\begin{equation}\label{resformel}
\mathring M^G_{k+1}= 
\sum_{j=0}^k{{k}\choose{j}} M^G_{j+1}\w \omega_\alpha^{k-j}.
\end{equation}
\end{lma}

\begin{proof}
Notice that 
\begin{equation}\label{putt5}
\log|G|^2=\log|sg\alpha|-\log|s\alpha|=\log|g\alpha|_\circ-\log|\alpha|^2_\circ
\end{equation}
We proceed by induction.
Notice that the case $k=0$ of \eqref{snorig} is trivial. Assume that it is proved for
some $k$. Together with  
  \eqref{putt5} and the recursion formula for $[dd^c\log|G|^2]^\ell$, cf.~\eqref{ma},  
\begin{multline*}
[dd^c\log|g\alpha|^2_\circ]^{k+1} =dd^c\big((\log|G|^2+\log|\alpha|^2_\circ) \langle dd^c\log|g\alpha|^2_\circ\rangle^{k}\big)=\\
\sum_{j=0}^k{{k}\choose{j}}[dd^c\log|G|^2]^{j+1}\w \omega_\alpha^{k-j} +
\sum_{j=0}^k{{k}\choose{j}}\langle dd^c\log|G|^2\rangle^{j}\w \omega_\alpha^{k-j+1}.
\end{multline*}
If we apply $\1_{Z'}$ to this relation we get \eqref{resformel} for $k+1$. 
If we apply $\1_{\Pk(E)\setminus Z'}$ we get \eqref{snorig} for $k+1$.  
Thus the lemma is proved. 
\end{proof}

There are several formulas for regularization of $M^G_k$.
For instance, see \cite[Proposition~5.7]{aeswy1},
$$
M^G_{k,\e}=\frac{\e}{([G|^2+\e)^{k+1}}(dd^c|G|^2)^k, \quad k=0,1,2,\cdots.
$$
By  
\eqref{resformel} we therefore  get global smooth $M^{g,\e}$ such that $M^{g,\e}\to M^g$. 
 Clearly, $\1_{Z'}[dd^c\log|g\alpha|_\circ^2]^{0}=\1_{Z'}=M^G_0$. 
In view of \eqref{defseg}, Definition~\ref{mdef} and  Lemma~\ref{snorkel} 
there are  non-negative integers $c_{j,k}$ such that 
\begin{equation}\label{resformel2}
M^g_\e=\sum_{k=0}^\infty \sum_{j=0}^\infty c_{j,k} p_*\Big(M^G_{k,\e}\w \omega_\alpha^j 
\Big)
\end{equation}
is a sequence of smooth forms that tends to $M^g$.

\vspace{.5cm} 

\section{Behaviour of $M^g$ under general proper mappings}\label{behaviour}

We have the following extension of Theorem~\ref{thmett}~(iii).

\begin{prop}\label{kakadua}
Let $g\colon E\to F$ be a morphism on $X$. Then $M^g$ induces a mapping $\mu\mapsto M^g\w \mu$
on $\GZ(X)$ and $\mathcal B(X)$ and if $h\colon X'\to X$ is any proper holomorphic mapping, then
\begin{equation}\label{proph}
h_*(M^{h^*g}\w \mu)=M^g\w h_*\mu
\end{equation}
for all $\mu\in \GZ(X')$ and $\mu\in\mathcal B(X')$. 
\end{prop}

\begin{ex}
If $h$ is a finite mapping, say generically $m$ to $1$, then we can apply \eqref{proph} to the function $\mu=1$. 
It follows that
$
h_*M^{h^*g}=mM^g.
$
\end{ex}


\begin{proof}[Proof of Proposition~\ref{kakadua}]
If  $\tau\colon W\to X$ is proper, then we 
have the commutative diagram
\begin{equation}\label{dia22}
\begin{array}[c]{cccc}
\Pk(\tau^*E)&  \stackrel{\tilde \tau}{\longrightarrow}  & \Pk(E) \\
\downarrow \scriptstyle{\tilde p} & &  \downarrow \scriptstyle{p}  \\
W & \stackrel{\tau }{\longrightarrow}  & X
\end{array}
\end{equation}
In fact, in a local trivialization $E=X\times \C^r_\alpha$ and $\tau^*E=W\times \C_\alpha^r$,
so that $\Pk(E)=X\times \Pk(\C^r_\alpha)$ and $\Pk(\tau^*E)= W\times \Pk(\C^r_\alpha)$. 
Assume that $\gamma$ is a product of first Chern forms and let $\mu=\tau_*\gamma$.
Since $p$ is a proper submersion the pullback $p^*\mu$ exists.  We claim that
\begin{equation}\label{volvo}
p^*\mu=\tilde\tau_*\tilde p^* \gamma.
\end{equation}
The equality \eqref{volvo} means that
\begin{equation}\label{volvo1}
p^*\mu.\xi=\tilde\tau_*\tilde p^* \gamma.\xi, \quad \xi\in\E(\Pk(E)).
\end{equation}
The left hand side of \eqref{volvo1} is, by definition, $\mu.p_*\xi$ which in turn is
$$
\mu(x).\int_\alpha \xi(x,\alpha)= \gamma(w).\int_\alpha \xi(\tau(w),\alpha)
$$
The right hand side is 
$$
\gamma.\tilde p_*\tilde\tau^*\xi=\gamma(w). \int_\alpha \xi(\tau(w),\alpha)
$$
as well. Thus \eqref{volvo} holds. In particular, $p^*\gamma$ is in $\GZ(\Pk(E))$.
%
Since $M^{g,\e}=p_* \big(s(L)\w \mathring M^{g\alpha,\e}\big)$, cf.~Proposition~\ref{mreg1}, 
$$
M^{g,\e}\w \mu=p_*\big(s(L)\w  \mathring M^{g\alpha,\e}\w p^*\mu\big).
$$
Since $p^*\mu$ is in $\GZ(\Pk(E))$  we can take limits, following the proof of \cite[Theorem~5.2]{aeswy1},  and get
\begin{equation}\label{kamel}
M^g\w \mu:=p_*\big(s(L)\w  \mathring M^{g\alpha}\w p^*\mu\big).
\end{equation}
We can extend by linearity to a general $\mu$. It is clear from \eqref{kamel} that this definition only depends on
$\mu$ and not on its representation. 
One must also check that if $\mu'\sim \mu$, then $M^g\w\mu\sim M^g\w\mu'$, but we omit the details.
 The equality \eqref{proph} follows from the corresponding property for $\mathring M^{g\alpha}$,
again  see \cite[Theorem~5.2]{aeswy1}, following the proof of Theorem~\ref{thmett}~(iii) above. 
\end{proof}

\vspace{.5cm}

\section{Vanishing of multiplicities}\label{treproofapa}

\begin{thm}\label{musiu}
Any $\mu\in \GZ_{n-k}(X)$ has a unique decomposition 
\eqref{siu}, where each irreducible component of
$N$ has Zariski support on a set of codimension $\le k-1$. The multiplicities of
$N$ vanish outside an analytic set of codimension $\ge k+1$.
\end{thm}

Since $\mu$ has a unique decomposition in irreducible components, the theorem follows 
from:  

\begin{prop}\label{surpuppa1}
If $\mu\in \GZ_{n-k}(X)$ is irreducible with Zariski support $Z$ and $\codim Z\le k-1$, then $\mult_x\mu$ 
vanish outside an analytic subset of $Z$ of codimension $\ge k+1$.
\end{prop}

In view of \cite[Remark~3.10]{aeswy1}, an irreducible $\mu$ as in Proposition~\ref{surpuppa1} is a
finite sum of $(k,k)$-currents $\tau_*\gamma$, where $\tau\colon W\to X$ and $\tau(W)=Z$.
If $\tau=i\circ \tau'$, where $i\colon Z\to X$, then $\mult_x \tau'_*\gamma=\mult_{i(x)}\tau_*\gamma$,
see Section~\ref{gztrams}. 
It is therefore enough to consider a surjective mapping $\tau\colon W\to Z$ and prove that if
$\mu=\tau_*\gamma$ has bidegree $(\ell,\ell)$ on $Z$, $\ell\ge 1$, then 
the subset of $Z$ where $\mult_x\mu\neq 0$ is contained in an analytic subset of codimension
$\ge \ell+1$.  Now Proposition~\ref{surpuppa1} follows from 
Lemma~\ref{surpuppa2} and  Proposition~\ref{surpuppa3} below.


\begin{lma}\label{surpuppa2}
Assume that $\tau\colon W\to Z$ is proper and surjective and $\mu=\tau_*\gamma$ has bidegree $(\ell,\ell)$.  
Let $r=\dim W-\dim Z$.
If $\mult_x\mu\neq 0$, then $\dim \tau^{-1}(x)\ge r+\ell$. 
\end{lma}

\begin{proof}
Let $n=\dim Z$ and 
let $\xi$ be a tuple that defines
the maximal ideal at $x$. Then, \cite[Section~6, Eq.~(6.1)]{aeswy1}, 
$$
\mult_x\mu [x]= M^\xi_{n-\ell}\w\tau_*\gamma= \tau_*\big(M^{\tau^*\xi}_{n-\ell}\w\gamma\big).
$$
 If this is non-vanishing, then since $\gamma$ is smooth, $M^{\tau^*\xi}_{n-\ell}$ is
non-vanishing. It has support on $\tau^{-1}(x)$ and therefore 
$
n-\ell\ge \codim_W \tau^{-1}(x)=n+r -\dim \tau^{-1}(x).
$
\end{proof}

The following proposition should be well-known but as we did not find a precise reference we 
provide a proof, cf.~Remark~\ref{gamas}. 

\begin{prop}\label{surpuppa3}
If $W$ is irreducible, $f\colon W\to Z$ is a proper surjective mapping and
$r=\dim W-\dim Z$, then
for each $\ell\ge 1$, the set
$$
A^f_{r+\ell}:=\{x; \ \dim f^{-1}(x)\ge r+\ell\}
$$
is contained in an analytic subset of codimension $\ge \ell+1$ in $Z$. 
\end{prop}

\begin{proof}[Proof of Proposition~\ref{surpuppa3}]
We can assume that $W$ is smooth, because otherwise we  take a regularization $\pi\colon W'\to W$ and
consider $f'=f\circ\pi$, noticing that 
$$
\{x; \ \dim f^{-1}(x)\ge r+\ell\}\subset \{x; \ \dim (f\circ\pi)^{-1}(x)\ge r+\ell\}.
$$

We proceed by induction over $\dim W$.
Assume that the proposition holds for all  $W$ with dimension $\le m$ and  $r$ such that $0\le r\le \dim W$,
and that our $W$ has dimension $m+1$. 
We first consider the case when $r=m+1$. Then all the sets $A^f_{r+\ell}$ for $\ell\ge 1$ are empty.
Thus we can assume from now on that $r\le m$. 
Notice that the set $W'\subset W$ where $\partial f/\partial w$ does not have optimal rank 
is analytic of dimension $\le m$.  

Moreover, observe that if $w\in W\setminus W'$, then $\partial f/\partial w$ has the same rank
in a \nbh of $w$ so by the 
constant rank theorem, there is a \nbh $U$ of $w$ such that 
$f^{-1}(f(w))\cap U$ has dimension $r$. 

Let    
$W'_j$ be the irreducible components of $W'$ and let
$f'_j$ be the restriction of $f$ to $W_j'$ so that
$f_j'\colon W'_j\to f(W'_j)$.
Since $f$ is proper, each $f(W'_j)$ is an analytic set.
We claim that
\begin{equation}\label{tyll}
A^f_{r+\ell}=\cup_j A^{f_j'}_{r+\ell}.
\end{equation}
In fact, assume that $f^{-1}(x)$ has an irreducible component $V$ of dimension $\ge r+\ell$. 
From the observation above it follows that a generic point on $V$ belongs to $W'$,
and hence $V$ is contained in $W'$. 
Thus 
$V=\cup _j V\cap W_j'.$
It follows that at least one of the analytic sets $V\cap W_j'$ has dimension $\ge r+\ell$. 
Thus $(f'_j)^{-1}(x)$ has dimension $\ge r+\ell$ so that $x\in A^{f'_j}_{r+\ell}$. 
Now \eqref{tyll} follows.

 \smallskip
In view of \eqref{tyll} it is enough to 
consider each $A^{f_j'}_{r+\ell}$. 
Assume that  $(f'_j)^{-1}(x)$ has generic dimension $r+\ell'$.
By definition then 
$$
\rank f'_j=\dim W'_j-r-\ell'
\le \dim W-1-r-\ell'=\dim W-1 -(\dim W-\dim Z)- \ell'=\dim Z-\ell' -1.
$$
Proposition~\ref{katt}  implies that 
\begin{equation}\label{katta}
\codim f'_j(W'_j) \ge \ell'+1.
\end{equation}

First assume that $\ell'\ge \ell$. Since $A^{f_j'}_{r+\ell}\subset f'_j(W'_j)$, by \eqref{katta},  
$$
\codim A^{f_j'}_{r+\ell}\ge \ell'+1\ge \ell+1
$$
as desired. 
Now assume that $\ell'<\ell$.  Since
$$
A^{f'_j}_{r+\ell}=A^{f'_j}_{r+\ell' +\ell-\ell'}
$$
it follows from the induction hypothesis that $A^{f'_j}_{r+\ell}$ is contained in an analytic subset of $f'_j(W'_j)$ 
of codimension $\ge \ell-\ell'+1$.  In view of \eqref{katta} we conclude that this analytic set has at least codimension
$\ell-\ell'+1+\ell'+1=\ell+2$ in $Z$.
%
%
%
Thus Proposition~\ref{surpuppa3} is proved.
\end{proof}

\begin{remark}\label{gamas}
If $\gamma$ in the proof of Lemma~\ref{surpuppa2} is strictly positive, then 
the multiplicity is strictly positive if and only
if $\dim \tau^{-1}(x)\ge r+\ell$.
If $W$ in Proposition~\ref{surpuppa3} has a K\"ahler form $\omega$, then $\gamma_\ell=\omega^\ell$ are strictly positive 
closed forms for $1\le \ell \le \dim W$. 
In this case therefore  Proposition~\ref{surpuppa3} follows from Siu's theorem applied to the 
positive closed currents  $f_*\gamma_\ell$.
\end{remark}


\vspace{.5cm}

\section{An extension of Theorem~\ref{thmB}}\label{plakat}

Let $g\colon E\to F$ be a morphism and let $a\colon s(E/\Ker g)\to \Im g$ be the induced isomorphism
over $X\setminus Z_0$.  Here is an extended version of Theorem~\ref{thmB}.

\begin{thm}\label{thmtre}
The natural extensions  $\1_{X\setminus Z_0}s(E/\Ker g)$ and $\1_{X\setminus Z_0}s(\Im g)$ are locally integrable 
and closed in $X$, and there is a current $M^a$, which is locally a generalized cycle,  with support on
$Z_0$ such that the following holds:

\smallskip
\noindent
(a) There is a current $W^a$ with singularities of logarithmic type along $Z_0$ such that
\begin{equation}\label{trelikhet1}
dd^c W^a=M^a+\1_{X\setminus Z_0}s(\Im g) -\1_{X\setminus Z_0}s(E/\Ker g).
\end{equation}
The analogue of  Theorem~\ref{thmett}~(ii) holds for $M^a$.  

\smallskip
\noindent
(b) If $\pi\colon\widetilde X\to X$ is a modification, and 
$M^{\pi^* a}$  denotes the current obtained from $\pi^*g$, then 
\begin{equation}\label{kamomilla}
\pi_* M^{\pi^*a}=M^a.
\end{equation}

\smallskip
\noindent
(c) If $\Ker g$ has an extension to a subbundle $N$ of $E$, and $a'$ is the induced extension to a morphism
$a'\colon E/N\to F$, then 
$
M^a=M^{a'},
$
where $M^{a'}$ is the current in Theorem~\ref{thmett}.

\smallskip
\noindent
(d)  All  multiplicities   $\mult_x M^a_k$ are integers.
There is a unique decomposition of the form \eqref{king},  where
$\mult_x N^a_k$ vanishes outside an analytic set of codimension $\ge k+1$.
All the coefficients $\beta_j^k$ are integers. 
If $g$ and $\hat g$ are comparable,
then the associated $M_k^a$ and $M_k^{\hat a}$ have the same multiplicities. 
\end{thm}

\noindent  Some of the multiplicities $\mult_x M_k^a$ and coefficents $\beta_j^k$ may be negative,  
see Example~\ref{katt2}. 

\smallskip\noindent
If $\pi\colon \widetilde X\to X$ is a modification such that $\pi^* \Ker g$ has an extension as a subbundle $\widetilde N$ 
over $\widetilde X$, such a modification always exists at least locally,  and $\tilde a$ denotes the induced
morphism $\pi^* E/\widetilde N\to  \pi^*F$, then 
$M^a=\pi_* M^{(\pi^* a)'}$
in view of (b) and (c) above.  Thus $M^a$ is determined by these properties.

\smallskip 
Although $a$ is only defined on $X\setminus Z_0$,  we can define smooth forms 
$M^{a,\e}$ on $X$ by \eqref{premium}.

 \begin{prop}\label{kork}
 The limit
\begin{equation}\label{adef} 
M^a:=\lim_{\e\to 0} M^{a,\e}
\end{equation}
exists and is independent of the choice of $\chi$ in \eqref{premium}. 
\end{prop}

If the subbundle $\Ker g\subset E$
defined in $X\setminus Z_0$ happens to have an extension as a subbundle $N$
of $E$ over $X$, then by continuity $N\subset \Ker g$ and therefore $g$ induces a 
morphism $a\colon E/N\to F$. 
By  Proposition~\ref{mreg2} then \eqref{adef} is consistent with the previous
definition of $M^a$.

\begin{proof}
Assume that $\pi\colon X'\to X$ is a modification such that the subbundle $\pi^*N\subset \pi^* E$  
on $X'\setminus \pi^{-1}Z_0$ extends to a subbundle $N'$ of $E'=\pi^*E$ on $X'$
Let $g'=\pi^*g\colon E'\to F'=\pi^* F$. Then 
$N'\subset \Ker g'$ and so $g'$ induces a generically injective mapping
$a'\colon E'/N'\to F'$.  Thus $M^{a'}$ is a well-defined current on $X'$. By
Lemma~\ref{ormen}, and its proof,
$M^{a,\e}=\pi_* M^{a',\e}$
and so
\begin{equation}\label{blurp}
M^a=\pi_* M^{a'}.
\end{equation}
In particular, it is independent of the choice of $\chi$. 
At least locally in $X$ such a modification $\pi$ exists, cf.~Section~\ref{extsub}, and thus the proposition is proved.
\end{proof}

\begin{proof}[Proof of Theorem~\ref{thmtre}]
Let $W^a$ be the form \eqref{wdef} but associated with $a$ rather than $g$ in $X\setminus Z_0$. Then 
$W^{a,\e}:=\chi_\e W^a$ is well-defined in $X$ for each $\e>0$. We claim that 
\begin{equation}\label{wa}
W^a:=\lim_{\e\to 0}W^{a,\e}
\end{equation}
exists.  To see this let $\pi\colon X'\to X$ be a modification as in the previous proof. If $\chi'_\e=\pi^*\chi_\e$, then 
$W^{a',\e}=\chi'_\e W^{a'}=\pi^* W^{a,\e}$.
Thus  
$W^{a,\e}=\pi_*W^{a',\e}$  and hence the limit \eqref{wa} exists and 
\begin{equation}\label{wa1}
W^a=\pi_*W^{a'}. 
\end{equation}
By Theorem~\ref{thmett}~(i)  we have 
\begin{equation}\label{polyp}
dd^cW^{a'}=M^{a'} +s(\Im g')-s(E'/N').
\end{equation}
Notice that 
$\1_{X\setminus Z_0}s(E/\Ker g)=\1_{X\setminus Z_0}s(E/N)=\pi_*s(E'/N')$ and 
$\1_{X\setminus Z_0}s(\Im g)=\pi_*s(\Im a')$ are locally integrable and closed.
Taking $\pi_*$ now  \eqref{trelikhet1}  follows from 
 \eqref{blurp}, \eqref{wa1}  and \eqref{polyp}.
Part (b) of the theorem follows in a standard way by choosing a $\pi$ as above which in addition factorizes over 
$\widetilde X$.
We omit the details.
Part (c) follows from the proof of (a).

Since $M^a$, at least locally, is a generalized cycle, all its multiplicities are integers, 
and we have the unique decomposition \eqref{king}, cf.~Section~\ref{gztrams}.

 If $g$ and $\hat g$ are comparable,  then $\pi^* g$ and $\pi^*\hat g$ are comparable in $X'$ and
hence the associated $a'$ and $\hat a'$ are comparable in $X'$.  It follows from the proof
of Theorem~\ref{thmett}~(vii) that 
$M^{a'}$ and $M^{\hat a'}$ belong to the same class in $\mathcal B(X')$. In view of \eqref{blurp} therefore 
$M^a_k$ and $M^{\hat a}_k$  belong to the same $\mathcal B$-class and hence they have 
the same multiplicities.  
Thus Theorem~\ref{thmtre} is proved.
\end{proof}

\begin{remark}
The non-negativity of the multiplicities in Theorem~\ref{thmett} was proved by locally 
choosing  trivial metrics locally on $X$ on $E$ and $F$.  This argument breaks down
for $M^a$ since it is the push-forward of
$M^{a'}$ under a  modification, and in general one cannot choose a metric locally on $X$ so that 
$M^{a'}$ is non-negative on  the exceptional divisor, cf.~Example~\ref{katt3}.
\end{remark}


 

 \vspace{.5cm}
 
\section{Chern and Segre forms associated with certain singular metrics}\label{BC-section} 
Singular metrics on line bundles have played a fundamental role in algebraic geometry during the last decades,
starting with \cite{Dem1}.
Singular metrics on a higher rank bundle were introduced in \cite{BePa}, see also \cite{Cataldo}, and have been studied by several authors since then, 
e.g., in \cite{Hos} and \cite{XM}. In \cite{Raufi} and later on in \cite{TI, LRS, LRSW} are introduced associated Chern forms.
In \cite{LRS} quite general singular metrics are allowed, but there are restrictions on the degrees.
In \cite{LRSW}  the whole Chern forms for metrics with analytic singularities are
defined; however in situations that go beyond  \cite{LRS} an a~priori choice of a smooth metric form is needed.  
These Chern forms are as expected where the metric is non-singular
and represent the de~Rham cohomology classes. 
We will use Theorems~\ref{thmett} and \ref{thmtre} to provide Chern and Segre forms,
that in addition represent the expected Bott-Chern classes, 
for two classes of singular metrics.

\begin{df}\label{platt}
Let $\hat F\to X$ be a holomorphic vector bundle with a metric that is non-singular outside an analytic set $Z$ of positive codimension. 
We say that a current  $s(\hat F)$ on $X$
is a Segre form for $\hat F$ if it represents the
Bott-Chern class  $\hat c(F)$ and is equal to the Segre form defined by the metric where it is non-singular.  We have the analogous definition of 
$c(\hat F)$.
\end{df}


\begin{ex}\label{lrsw}
Let $E$ and $F$ be Hermitian vector bundles and $g\colon E\to F$ a holomorphic morphism. 
Let $\hat E$ be $E$ but equipped by the singular metric so that
$|s|_{\hat E}=|gs|$.  It was proved in \cite{LRSW} that,
  in our notation,  the current
$$
s(\hat E)=M^g+\1_{X\setminus Z}s(\Im g),
$$
defines the same de~Rham cohomology class as $s(E)$. 
Theorem~\ref{thmett}~(i) states that it in fact defines the same Bott-Chern class, 
so that $s(\hat E)$ is a Segre form for $\hat E$ in the sense of Definition~\ref{platt}.
If $g$ is generically surjective it follows from the proof of Proposition~\ref{avig} that  
\begin{equation}\label{kvast}
c(\hat E)=-c(E)c(F)M^g+c(F) 
\end{equation}
is a Chern form for $\hat E$.  
Notice that the multiplicities of $s(\hat E)$ and $-c(\hat E)$ coincide and are independent of the smooth metrics on $E$ and $F$.
\end{ex}

\begin{remark}
One can obtain an analogue of \eqref{kvast} for an arbitrary $g$; for simplicity though we assume that $Z$ has positive codimension.
Using the ideas in the proof of Proposition~\ref{hummer} below one can define a current
$M^{g,b}$  and a locally integrable $V^g$ such that
$dd^c V^g=-M^{g,b}+\1_{X\setminus Z}c(\Im g)-c(E)$,  so that
$c(\hat E)=-M^{g,b}+\1_{X\setminus Z}c(\Im g)$ is a Chern form for $\hat E$.
%
\end{remark}

In our second example we assume that
 $g\colon E\to F$ is a generically surjective morphism, $E$ and $F$ Hermitian vector bundles,  and we let $\hat F$ be $F$ but equipped with the singular metric induced from $E$. 
That is, for $\beta\in F$ and $x\in X\setminus Z_0$,  
$|\beta|_{\hat F}=|g^{-1}\beta|_{E/\Ker g}$.  
Then clearly $\hat F$ is isometric to $E/\Ker g$ in $X\setminus Z_0$ so that $s(\hat F)=s(E/\Ker g)$  and $c(\hat F)=c(E/\Ker g)$ there. 

\begin{prop}\label{hummer}
With the notation in Theorem~\ref{thmtre},  
\begin{equation}\label{fhat}
s(\hat F)=\1_{X\setminus Z_0}s(E/\Ker g)-M^a
\end{equation}
is a Segre form for $\hat F$. There is a  related current  $M^{a,c}$ with support on $Z_0$
such that
\begin{equation}\label{fhat1}
c(\hat F)=\1_{X\setminus Z_0}c(E/\Ker g)+ M^{a,c}
\end{equation}
is a Chern form for $\hat F$.
The multiplicities of $M^a$ and $M^{a,c}$ are independent of 
the smooth metrics on $E$ and $F$. 
\end{prop}


%
\begin{cor}
If $g$ is generically an isomorphism and $E$ is trivial with a trivial metric, then
\begin{equation}\label{hackspett}
s(\hat F)=1-M^g, \quad  c(\hat F)=1+M^{g,c}.
\end{equation}
\end{cor}

Different trivial metrics on $E$ may produce different $M^g$, 
see Examples~\ref{snurr1} and \ref{snurr2}.  However, 
$-s_1(\hat F)=c_1(\hat F)=[\det g]$, see Proposition~\ref{m1} and the definition of $M^{g,c}$ below.

\begin{proof}[Proof of Proposition~\ref{hummer}]
Clearly \eqref{fhat} is equal to $s(E/\Ker g)$ in $X\setminus Z_0$.  
Theorem~\ref{thmtre}~(a) implies that \eqref{fhat} is 
Bott-Chern cohomologous with $s(F)$, and thus a Segre form for $\hat F$.

Let $\pi\colon X'\to X$ be a modification as in the proof of Theorem~\ref{thmtre}.
Then we have, cf.~\eqref{polyp},
$ 
dd^cW^{a'}=M^{a'} +s(F')-s(E'/N').
$ 
Since $s(E'/N')$ and $c(E'/N')$ are smooth, we get 
\begin{equation}\label{kobra}
dd^c V^{a'}=M^{a',c}+c(E'/N')-c(F'),
\end{equation}
where $M^{a',c}= c(F')c(E'/N') M^{a'}$ and $V^{a'}=c(F')c(E'/N') W^{a'}$.
We define
\begin{equation}\label{pudel}
M^{a,c}=\pi_*M^{a',c}, \quad V^a=\pi_* V^{a'}. 
\end{equation}
By regularization as in the proof of Theorem~\ref{thmtre} one verifies that the definitions in \eqref{pudel} are independent of $\pi$. 
Thus $M^{a,c}$ and $V^a$ are globally defined on $X$.
Applying  $\pi_*$ to \eqref{kobra} we get
$$
dd^c V=M^{a,c} +\1_{X\setminus Z_0}c(E/\Ker g)-c(F).
$$
Thus \eqref{fhat1} is a Chern form for $\hat F$. 

The class of the current $M^a$ in $\mathcal B(X)$ is independent of the smooth metrics on $E$ and $F$.
The same holds for the class of  $M^{a',c}$ in $\mathcal B(X')$ and hence for the class
of $M^{a,c}$ in $\mathcal B(X)$.  Thus the statements about $\mult M^a$ and $\mult M^{a,c}$ follow.
\end{proof}


%

\vspace{.5cm}

\section{Segre numbers and distinguished varieties associated with a coherent sheaf}\label{segresection}

These numbers, which generalize
the Hilbert-Samuel multiplicity of $\J_x$, were introduced, with a geometric definition,
in the '90s,  independently by Tworzewski, \cite{Twor} and Gaffney-Gassler,  \cite{GG}.  Later on
a purely algebraic definition was given in Achilles-Manaresi, \cite{AM},  and Achilles-Rams, \cite{AR}.  
We can consider such a  $g$ as a morphism $E\to F$, where $E=X\times \C$ is a trivial line bundle
with a trivial metric. 

\bigskip

Assume that $g$ is a holomorphic section of a vector bundle $F$, that is, $E$ is trivial line bundle in our
set-up.  Then $g$ generates an ideal sheaf $\J\subset\Ok$ which is precisely the image of the 
dual morphism  $g^*\colon \Ok(F^*)\to \Ok(E^*)=\Ok$.  The decomposition 
\eqref{king} is a generalization of the classical King formula, \cite{King}, and the analytic sets $Z_j^k$ that
appear in the fixed part are precisely the so-called distinguished varieties associated with $\J$.  If
$\pi\colon X'\to X$ is the blow-up of $X$ along $\J$, then $Z_j^k$ are precisely the images of the various
irrreducible components of the exceptional divisor in $X'$.  As mentioned in the introduction, the 
multiplicities $\mult_x M^g_k$ are the so-called Segre numbers $e_k(\J_x)$ of $\J_x$.  

We will discuss generalizations 
to arbitrary coherent (analytic) sheaves.
As for notions like Cohen-Macaulay, dimension etc, we 'identify'  an ideal sheaf $\J$ with the quotient sheaf $\Ok/\J$. 
By definition an arbitrary coherent sheaf $\F$ locally has a representation 
$\F=\Ok(E^*)/\Ims g^*$, where $g\colon E\to F$ is a holomorphic morphism. 

\begin{prop}\label{bamba}
Given a coherent sheaf $\F=\Ok(E^*)/\Ims g^*$, the multiplicities $\mult_x M^g_k$ and the
fixed part of the decomposition \eqref{king} only depends on $\F$.
\end{prop}

Taking this proposition for granted the following definitions may be reasonable.

\begin{df}
If the coherent $\F$ has the local presentation $\F=\Ok(E^*)/\Ims g^*$, then we define its Segre
numbers $e_k(\F_x)=\mult_x M^g_k$, $k=0,1, \ldots$, and its distinguished varieties as
the various components of the fixed part in \eqref{king} for various $k$. 
\end{df}

It follows from Theorem~\ref{thmett} that the Segre numbers $e_k(\F_x)$ are non-negative integers that can be strictly positive only if
$x\in Z$ and $k\ge \codim Z$.  

\begin{remark}\label{brjox}
If $\F$ has zero set $\{x\}$, then the Buchsbaum-Rim multiplicity was  introduced in
\cite{BR}.  This definition is algebraic, but a geometric description appeared
in  \cite{KT1,KT2} and \cite{HM}. One can verify that it indeed coincides with 
$\mult_x M^g_n$.  A detailed argument will be given in a forthcoming paper.
If the singularity is not isolated, in  \cite{CP} is defined algebraically a list of numbers, generalizing
the description in \cite{AM} of Segre numbers in case of an ideal. 
One could guess that these numbers coincide with the numbers $\mult_x M^g_k$.
\end{remark}

Let $\pi\colon Y\to\Pk(E)$ be the blow-up of $\Pk(E)$ along $G=g\alpha$. 
In view of the discussion above and the proof of Theorem~\ref{thmett}~(viii),
the distinguished varieties of $\F$ are the images under $p\circ \pi$ of the various
irreducible components of the exceptional divisor of $\pi$.

\begin{proof}[Proof of Proposition~\ref{bamba}]
A minimal free resolution of $\F$ at a point $x$ is unique, up
 to biholomorphisms, and any resolution at $x$ is the direct sum of a minimal
 resolution and a resolution of $0$. The latter resolution ends with a pointwise surjective mapping
 $(g')^* \colon   (F)^*\to (E')^*$. 
 If $g^*$ is the last mapping in a minimal resolution of $\F$ at $x$, then
$\F=\Ok(E^*)/\Im g^*$ and  any other representation has the form 
 $$
 \F= \Ok(E^*\oplus (E')^*)/\Im (g^*\oplus(g')^*),
 $$
 where $g'\colon E'\to F'$ is pointwise injective.  
In view of Theorem~\ref{thmett}~(v) and Lemma~\ref{nolllma} thus  
\begin{equation}\label{multplus}
\mult_x M_k^{g\oplus g'}=\mult_x M^g_k, \quad k=0,1,2,\ldots.
\end{equation}
Thus these numbers are intrinsic for the sheaf $\F$ at $x$.  
Consider now the representation \eqref{king} for $M_k^{g\oplus g'}$ and $M^g_k$, respectively.
Since $N_k^{g\oplus g'}$ and $N_k^{g}$ only  have non-zero multiplicities on sets of
codimension $\ge k+1$, \eqref{multplus} implies that $M_k^{g\oplus g'}$ and $M^g_k$ have the same fixed part.
\end{proof}

\begin{ex}
The morphism $g^*(x)$ in  Example~\ref{hund} below gives the coherent sheaf
$
\F=\Ok\oplus \Ok/x_1\Ok\oplus x_2\Ok,
$
and it is shown that its distinguished varieties are the axes and the point $(0,0)$.
Moreover, it has  non-zero
multiplicities on both codimension $1$ and $2$.
The morphism defined by the matrix
\begin{equation} 
\begin{bmatrix} 
	x_1x_2 & 0 \\
	0 & 1 \\
\end{bmatrix}.
\end{equation}
gives the sheaf
$
\F=\Ok\oplus \Ok/x_1x_2\Ok\oplus \Ok=\Ok/x_1x_2\Ok,
$
which we identify with the ideal sheaf $\la x_1x_2\ra$. 
It has the coordinate axes as distinguished varieties and non-zero
multiplicities only on codimension $1$.
However, the determinant ideals in both cases are $\la x_1 x_2\ra$.
Thus neither distinguished varieties nor multiplicities can be computed from the determinant ideal.
\end{ex}

 \vspace{.5cm}

\section{Some examples and remarks}\label{exsection}

We will use the notation introduced in Section~\ref{defsec}.
We present our first example as a proposition.

\begin{prop}\label{m1}
If $g\colon E\to F$ is generically an isomorphism, then
\begin{equation}\label{drummel}
M_{1}^g=[\dv(\det g)].
\end{equation}
\end{prop}

\begin{proof}
Let $Z$ be the zero set of $\det g$.  Since  $M^g_1$ is a $(1,1)$-current with support on
the hypersurface $Z$ it must be (the Lelong current of) a cycle with support on $Z$.  
It is therefore enough 
to check, for any regular point $x\in Z$,  that $\mult_x M^g=\mult_x [\dv(\det g)]$.

Let us first assume that $n=1$,  that the base space $X$ is a \nbh $\U$ of the closed unit disk, 
 $E=\U\times \C^r$,  and  $F=\U\times \C^r$ and $\det g(x)=x^\nu a$ in $\U$, where $a$ is non-vanishing.
Since the multiplicities are independent of the metrics on $E$ and $F$ we can assume that they are trivial in $\U$.
If $\nu=1$, then $g(0)$ has a simple eigenvalue and hence a one-dimensional kernel. Thus $\mathring M^{g\alpha}_r$
is a point mass in $\Pk(E)$ and hence
$$
M^g_1=p_*(s(L)\w  \mathring M^{g\alpha}_r)=p_*\mathring M^{g\alpha}_r=[0] .
$$ 
Now assume that $\nu>1$. 
We can choose  a continuous perturbation $g_t$ of $g$ such that $g_0=g$ and
 $\det g_1$ has $\nu$ distinct simple zeros  $x^1,\ldots, x^\nu$ close to $x=0$.  
 Then the kernel of each  $g(x^j)$ is one-dimensional, so that 
 $M^{g_1}=[x^1]+\cdots +[x^\nu]$ and so
its total mass is $\nu$. 
Since we have trivial metrics $s_1(E)=0$ and $s_1(F)=0$, so by \eqref{main111},
 $$
\int_{|x|<1} M_1^{g_t}=\int_{|x|<1} dd^c W_0^{g_t}=\int_{|x|=1} d^c W_0^{g_t}.
$$
For each $t$ the integral is a sum of the Lelong numbers (multiplicities) of $M^{g_t}_1$ so by Theorem~\ref{thmett} it is a positive integer.
From formula \eqref{wdef} we see that $w^{g_t}$ depends continuously on $t$ on $|x|=1$. Thus
the integral is $\nu$ also for $g=g_0$, so the proposition  holds when $n=1$.  

Now assume that $n>1$, $0$ is a regular point on $Z$, and locally $Z=\{x_1=0\}$. 
Then  $\det g=x_1^\nu a$ for some $\nu$ and non-vanishing holomorphic function $a$. 
From the discussion above we know that  $M^g_1=\mu[x_1=0]$ so we have to prove that $\mu=\nu$.
For a generic choice of complementary coordinate functions  $x_2,\ldots,x_n$,
$$
\mu=\mult_0 M_1^g=\mult_0 \big([x_2=\cdots =x_n=0]\w M_1^g\big).
$$
Let $i\colon \C_{x_1}\to \C^n_x, \  x_1\mapsto (x_1, 0, \ldots, 0)$.  By Proposition~\ref{kakadua} thus
\begin{equation}\label{kungar}
i_* M^{i^* g}_1= [x_2=\cdots =x_n=0]\w M_1^g=\mu[0].
\end{equation}
Now $\det i^*g(x_1)=a(x_1,0)x_1^\nu$ so from the case $n=1$ we have  $M^{i^*g}_1=\nu[0]$ in $\C$ so that
$i_* M^{i^* g}_1=\nu[0]$ in $\C^n$. In view of \eqref{kungar} thus $\mu=\nu$.
\end{proof}

We will %
use the following form of Crofton's formula, see, e.g., \cite[Lemma~6.3]{aswy}:
If $(f_1, \ldots, f_m)$ is a tuple of holomorphic functions and 
 $[\gamma_1,\ldots,\gamma_m]\in \Pk(\C^m_\gamma)$, then
\begin{equation}\label{fjant1}
\int_\gamma [\dv(\gamma_1f_1+\cdots+\gamma_mf_m)] d\sigma(\gamma)=dd^c\log(|f_1|^2+\cdots+|f_m|^2).
\end{equation}
Here $d\sigma$ is the normalized volume form associated with the Fubini-Study metric on $\Pk(\C^m)$. 
If in addition $\dv f_1,\ldots,\dv f_m$ intersect properly, i.e., the codimension of their intersection is $m$,
then  
$$
\big(dd^c\log(|f_1|^2+\cdots+|f_m|^2)\big)^k=\big[dd^c\log(|f_1|^2+\cdots+|f_m|^2)\big]^k
$$
is locally integrable for $k<m$ and
\begin{equation}\label{fjant}
M^f_m=[dd^c\log(|f_1|^2+\cdots+|f_m|^2)]^m=[\dv f_1]\w\cdots\w [\dv f_m].
\end{equation}
The right hand side is the (Lelong current of the) of the intersection product 
of the divisors and can be defined by any reasonable regularizations of the $[f_j]$, see \cite[2.12.3]{Ch}.
It is well-known that this product is unchanged if 
$f_j$ are replaced by $\gamma^j\cdot f=f_1\gamma_1^j+\cdots +\gamma_m^jf_m$ for generic choices of $\gamma^j\in \Pk(\C^m)$.
Therefore one can deduce  \eqref{fjant} from \eqref{fjant1}. 
In the examples below we often write $[f=0]$ for $[\dv f]$.

\begin{ex}\label{hund}
Let $X=\C_x^2$,  $E=X\times\C^2_\alpha$, $F=X\times \C^2$, both with trivial metric.
and $g\colon E\to F$ defined by
\begin{equation}\label{matrix}
\begin{bmatrix} 
	x_1 & 0 \\
	0 & x_2 \\
\end{bmatrix}.
\end{equation}
Then $g\alpha=(x_1\alpha_1,x_2\alpha_2)$  defines a proper intersection in $\C^2_x\times\Pk(\C^2_\alpha)$
so by \eqref{fjant}
\begin{equation}\label{setboll}
\mathring M^{g\alpha}=\mathring M^{g\alpha}_2=
[dd^c\log|g\alpha|^2_\circ]^2=[x_1=\alpha_2=0]+[x_2=\alpha_1=0]+[x_1=x_2=0].
\end{equation}
Since $s(L)=1+\omega_\alpha$ we see that $M^g=M^g_1+M^g_2$, where
$$
M^g_1=[x_1=0]+[x_2=0]=[x_1x_2=0], \quad M^g_2=[x_1=x_2=0].
$$
Notice that $M^g_1=[\dv(\det g)]$  in accordance with Proposition~\ref{m1}. 
\end{ex}


The next example shows that in general several components of $\mathring M^{g\alpha}$ come into 
play to produce $M^g_1$.

\begin{ex}
Let $X=\C$, $E=X\times \C_\alpha^2$ and $F=X\times \C^2$, and let both $E$ and $F$ be equipped with the trivial metric. 
Let $g\colon E\to F$ be the morphism defined by the matrix 
$$
\begin{bmatrix} 
	x^2 & 0 \\
	0 & x \\
\end{bmatrix}.
$$
Notice that $\det g=x^3$, $Z=\{0\}$, and $Z'=\{0\}\times \Pk(\C^2_\alpha)$. 
Now
$dd^c\log|g\alpha|^2_\circ=[x=0]+dd^c\log([x\alpha_1|^2+|\alpha_2|^2)_\circ$ so that
$$
\mathring M_1^{g\alpha}=\1_{Z'}dd^c\log|g\alpha|^2_\circ=[x=0].
$$
Furthermore, a computation using \eqref{fjant},  yields that
$
dd^c\big(\log|g\alpha|^2_\circ \1_{\Pk(E)\setminus Z'}dd^c\log|g\alpha|^2_\circ\big)=
[x=0]\w[\alpha_2]+[x\alpha_1=0]\w[\alpha_2=0]
$
and hence
$$
\mathring M_2^{g\alpha}=2[x=0]\w[\alpha_2=0].
$$
Altogether, as expected from Proposition~\ref{m1} 
$$
M^g_1=p_*\big(s(L)\w \mathring M^{g\alpha}\big)=p_*\big(\omega_\alpha\w[x=0]+2[x=0]\w[\alpha_2=0]\big)=3[x=0].
$$
%
%
\end{ex}

\begin{ex}\label{snurr1}
Let $X=\C^2$, $E=X\times \C^2_\alpha$, $F=X\times \C$ with trivial metrics, and $g$ the morphism given by
$[x_1 \ x_2]$. 
Since $g$ is not generically injective, $Z=X$. 
Moreover, $Z'=\{(x,[\alpha]); \ x_1\alpha_1+x_2\alpha_2=0\}$. We have 
$
\mathring M^{g\alpha}=\mathring M_1^{g\alpha}=[x_1\alpha_1+x_2\alpha_2=0].
$
Since $s(L)=1+\omega_\alpha$ we get, using \eqref{fjant1}, 
$$
M^g=M^g_0+M^g_1= \1_X + dd^c\log(|x_1|^2+|x_2|^2).
$$
Here $M^g_0$ is the fixed part, and it consists of the single distinguished  variety $X$.
The  term $M^g_1$ has dimension $1$ and is geometrically the mean value of lines through the origin in $X$,
so it is a moving term.
It follows that 
$\mult_{(0,0)}M^g_1=1$ but $\mult_{(x_1,x_2)}M^g_1=0$ for $(x_1,x_2)\neq(0,0)$. 

If we change the trivial metric on $E$, e.g., by letting $|\alpha|^2:=|\alpha_1|^2+2|\alpha_2|^2$, then 
$\omega_\alpha=dd^c\log(|\alpha_1|^2+2|\alpha_2|^2)$ and one can verify that then
$M^g_1=dd^c\log(2|x_1|^2+|x_2|^2)$. 
\end{ex}
 
Here is a similar  example  but where $g$ is generically injective.

\begin{ex}\label{snurr2} 
Let $X=\C^3$, $E=F=X\times \C^3$ with trivial metrics, and $g$ given by 
$$
\begin{bmatrix} 
	x_1x_3 & 0 & 0 \\
	0 & x_2x_3 & 0\\
        0& 0& x_3^2
\end{bmatrix}.
$$
Then $Z=\{x_1 x_2 x_3=0\}$, and
$g\alpha=x_3(x_1\alpha_1, x_2\alpha_2, x_3\alpha_3)$ so that 
$$
dd^c\log|g\alpha|_\circ^2=[x_3=0]+dd^c\log(|x_1\alpha_1|^2+|x_2\alpha_2|^2+|x_3\alpha_3|^2).
$$
Thus $\mathring M^{g\alpha}_1=[x_3=0]$. Next we have
\begin{multline*}
[dd^c\log|g\alpha|_\circ^2]^2=\\
dd^c\big(\log|x_3|^2+\log(|x_1\alpha_1|^2+|x_2\alpha_2|^2+|x_3\alpha_3|^2)\w 
dd^c\log(|x_1\alpha_1|^2+|x_2\alpha_2|^2+|x_3\alpha_3|^2)\big) =\\
[x_3=0]\w dd^c\log(|x_1\alpha_1|^2+|x_2\alpha_2|^2)+\big(dd^c\log(|x_1\alpha_1|^2+|x_2\alpha_2|^2+|x_3\alpha_3|^2)\big)^2.
\end{multline*}
Thus 
$$
\mathring M_2^{g\alpha}=[x_3=0]\w dd^c\log(|x_1\alpha_1|^2+|x_2\alpha_2|^2).
$$
Furthermore we get, using \eqref{fjant}, 
$$
\mathring M_3^{g\alpha}=[x_3=0]\w \big(dd^c\log(|x_1\alpha_1|^2+|x_2\alpha_2|^2)\big)^2+
[x_1\alpha_1=0]\w [x_2\alpha_2=0]\w[x_3 \alpha_3=0].
$$
We do not compute all terms of $M^g$ but notice that, e.g., 
$$
\int_\alpha [x_3=0]\w dd^c\log(|x_1\alpha_1|^2+|x_2\alpha_2|^2)\w \omega_\alpha
$$
is a non-zero term in $M^g_2$ that has support on the hyperplane $[x_3=0]$. 
As in Example~\ref{snurr1} one can verify that $M^g_2$ depends on the choice of trivial metric on $E$.
\end{ex}

Assume that $g\colon E\to F$ is generically an isomorphism. Then $g^*\colon F^*\to E^*$ is 
as well. In view of \eqref{main111} and the fact that 
\begin{equation}\label{pontus}
s_k(E^*)=(-1)^k s_k(E)
\end{equation}
it follows that 
$M^{g^*}_k$ and $(-1)^{k+1}M^g_k$ define the same Bott-Chern class.

\begin{remark}\label{taktik}
It is {\it not} true in general that $M^{g^*}_k=(-1)^{k+1}M^g_k$. In fact, given trivial metrics on $E$ and $F$
we know that both $M^g$ and $M^{g^*}$ are positive currents.  Therefore \eqref{pontus} fails as soon as
$M^g_k$ is non-zero for an even $k$. 
See, e.g., $M^g_2$ in Example~\ref{hund}.
\end{remark}

Let us now consider a global version of Example~\ref{hund}.  
\begin{ex}\label{hund1}
Let $X=\Pk^2=\Pk(\C_{x_0,x_1,x_2})$. Then $x_j$ is a section of $\Ok(1)\to X$ and thus 
defines a morphism 
$\Ok(-1)\to X\times \C$. 
If $E=X\times \Ok(-1)\otimes\C^2_\alpha$ and  $F=X\times \C^2_\alpha$
thus  \eqref{matrix} defines a morphism $g\colon E\to F$.  We choose the natural metric on $\Ok(-1)$ so that
$s_1(\Ok(-1))=dd^c\log|x|^2=\omega_x$ on $X$.  It follows that $L$ then is the tautological line bundle
with respect to trivial metric on $\C^2_\alpha$ tensored by $\Ok(-1)$, so that
$s_1(L)=\omega_\alpha+\omega_x$. Noting that $\Pk(E)$ has dimension $1$ in $\alpha$,  therefore
\begin{equation}\label{porter}
s(L)=1+\omega_\alpha+\omega_x+2\omega_\alpha\w \omega_x+\omega_x^2+3\omega_x^2\w\omega_\alpha.
\end{equation}
Applying $p_*$ to \eqref{porter} we get 
\begin{equation}\label{perro1}
s(E)=1+2\omega_x+3\omega_x^2.
\end{equation}
Since the metric on $F$ is trivial we see that \eqref{setboll} still holds in this case (but interpreted on $\Pk(E)$). 
Combined with \eqref{porter} we can compute  $M^g=p_*(s(L)\w \mathring M^{g\alpha})$ and find that
\begin{equation}\label{perro2}
M^g_1=[x_1=0]+[x_2=0],\quad M^g_2=\omega_x\w [x_1=0]+\omega_x\w[x_2=0]+[x_1=x_2=0].
\end{equation}
Notice that \eqref{perro1} and \eqref{perro2} are in accordance with  \eqref{main111} 
since  $M^g_1$ and $M^g_2$ are Bott-Chern cohomologous
with $2\omega_x$ and $3\omega_x^2$, respectively,  on $X=\Pk^2$ and $s(F)=1$. 
\end{ex}

\begin{ex}\label{hund2}
Let us consider the adjoint mapping $g\colon X\times \C^2_\alpha\to X\times \Ok(1)\otimes \C^2$. In this 
case $s(L)=1+\omega_\alpha$ and $s(E)=1$. Now
$$
|g\alpha|^2_\circ=(|x_1\alpha_1|^2+|x_2\alpha_2|^2)_\circ/|x|^2
$$
and so
$$
dd^c\log|g\alpha|^2_\circ=dd^c\log((|x_1\alpha_1|^2+|x_2\alpha_2|^2)_\circ-\omega_x.
$$
We see that 
$$
\mathring M^{g\alpha}_2= \1_{Z'}[dd^c\log|g\alpha|^2_\circ]^2=[x_1=\alpha_2=0]+[x_2=\alpha_1=0]+[x_1=x_2=0]
$$
as before, whereas
$$
\1_{\Pk(E)\setminus Z'}[dd^c\log|g\alpha|^2_\circ]^2=-2dd^c\log(|x_1\alpha_1|^2+|x_2\alpha_2|^2)_\circ\w \omega_x
+\omega_x^2.
$$
Therefore
\begin{multline*}
\mathring M^{g\alpha}_3= \1_{Z'}[dd^c\log|g\alpha|^2_\circ]^3=\\
-2\big([x_1=\alpha_2=0]+[x_2=\alpha_1=0]+[x_1=x_2=0]\big)\w \omega_x =\\
-2[x_1=\alpha_2=0]\w\omega_x-2[x_2=\alpha_1=0]\w\omega_x.
\end{multline*}
Recalling that $s(L)=1+\omega_\alpha$ we get
\begin{equation}\label{perro3}
M^g_1= [x_1=0]+[x_2=0], \quad M^g_2=[x_1=x_2=0]-2[x_1=0]\w\omega_x-2[x_2=0]\w\omega_x.
\end{equation}
In view of \eqref{pontus} and \eqref{perro1},  $s(F)=1-2\omega_x+3\omega_x^2$.
Thus, cf.~\eqref{perro3},  \eqref{main111} is respected. 
\end{ex}

\begin{ex}\label{kisse}
Let $X=\Pk^2_x$ and consider the morphism $g\colon \Ok(-1)\to  X\times \C_\alpha^2$, where $g=[x_1\ x_2]$, so that
$g$ is singular at the point $p=[1,x_1,x_2]$.  We see that 
$$
s_1(\Im g)=dd^c\log(|x_1|^2+|x_2|^2)=:\omega_p
$$
in $X\setminus\{p\}$. It follows that $s_2(\Im g)=\omega_p^2=0$ in $X\setminus\{p\}$. Since $M^g_2$ has support at $p$
it must be $\alpha[p]$ for some integer $\alpha$. 
By  \eqref{main11}, 
$$
dd^c W^g_1=\1_{X\setminus\{p\}}s_2(\Im g)-s_2(\Ok(-1))^2+M^g_2=-\omega^2+M_2^g
$$
so we conclude that $M_2^g=[p]$.   It also follows directly, cf.~\eqref{fjant}, that
$$
M^g_2=\1_{\{p\}}dd^c\big(\log(|x_1|^2+|x_2|^2) \1_{X\setminus\{p\}}dd^c\log(|x_1|^2+|x_2|^2))=[p].
$$
\end{ex}

We shall now see that the morphism $a$ in Theorem~\ref{thmtre} can have negative multiplicities.

\begin{ex}\label{katt2}
Let $X=\Pk^2$ and consider the morphism $g'\colon X\times \C^2_\alpha\to \Ok(1)$; it is the dual of the morphism
in Example~\ref{kisse}.  Consider the induced morphism
$$
a\colon \C^2\times \Pk^2/\Ker g' \to \Ok(1).
$$
From  \eqref{pontus} 
we see that
$s_2(E/\Ker g')=0$.  By \eqref{trelikhet1}  in Theorem~\ref{thmtre} therefore  
$$
dd^c W^a_1=\omega^2 + M_2^{a},
$$
so we can conclude that $M_2^{a}=-[p]$.
\end{ex}

Let us now make a direct computation that reveals how the minus sign in the previous example appears,
without relying on the global formula \eqref{trelikhet1}.  
We consider a somewhat more general mapping, but restrict to the local
situation.

\begin{ex}\label{katt3}
Let $X=\U\subset \C^2$, $E=X\times \C^2_\alpha$, $F=X\times \C$ (with trivial metrics) and $g=(g_1,g_2)$ with an isolated zero at $0\in\U$. 
Let $\pi\colon X'\to X$ be a modification such that $\pi^*g=g^0g'$,  where $g^0$ is a section of the line bundle $\La\to X'$
$g'=(g'_1,g'_2)$ is a non-vanishing section of $\La^*\otimes \C^2$.
The kernel of $\pi^*g$ is generated by $(-g_2', g_1')$ in $X'\setminus \pi^*(0)$ and it thus has
a holomorphic extension to a subbundle of $E'=\pi^* E$ over $X'$.  Notice that the image  in 
 $E'/N'$  of  the holomorphic section $u_1=(1,0)$ is non-vanishing in the open subset of $X'$ where  $g'_1\neq 0$.  
 The norm of the image of $u_1$ in $E'/N'$ is the $E'$-norm of
$$
\hat u_1=u_1-\frac{u_1\cdot (-\bar g_2',\bar g_1')}{|g'|^2}(-g_2', g_1').
$$
A straight forward computation reveals that $|\hat u_1|^2=|g_2'|^2/|g'|^2$, and thus
$dd^c\log |\hat u_1|^2= -dd^c\log|g'|^2$ in the set where $g'_1\neq 0$. An analogous  formula holds where $g'_2\neq 0$.
Since $E'/N'$ is a line bundle we conclude that 
$$
s_1(E'/N')=-dd^c\log|g'|^2, \quad s_2(E'/N')=(-dd^c\log|g'|^2)^2=0.
$$
Notice that  $a'\colon E'/N'\to \pi^* F$ is defined by $a'=g^0(g'_1,g'_2)$ so that 
$\dv a'=[g^0=0]$.  Recalling, cf.~\eqref{ragnar},  that 
$M^{a'} =s(E'/N')\w [\dv a']$ we thus have 
$
M^{a'}_1=[g^0=0] $ and $M^{a'}_2= s_1(E'/N')\w [g^0=0].$
We conclude that
$$
M_2^a=\pi_* M_2^{a'}=-c[0],
$$
where $c$ is a positive integer. In fact,  $M^{g^*}_2=c[0]$ so $c$ is the multiplicity of the zero of $g$ at $0$. 
\end{ex}

\begin{remark}\label{apl}
Let $E$ be a trivial line bundle (with trivial metric) and let $g\colon E\to F$ be generically injective morphism, i.e., 
a non-trivial  holomorphic section of
$F$. With the notation in this paper a residue current, here denoted by $M^{g,a}$,  was defined in \cite{Apl} in the 
following way. Let $S$ denote $E$ but with the singular metric inherited from $F$. Then $\1_{X\setminus Z} c(F/S)$ is 
locally integrable in $X$ and 
$$
M^{g,a}:=\1_Z dd^c(\log|g|^2 c(F/S)).
$$
If $\pi\colon X'\to X$ is a suitable modification, then $\pi^*c(S)$ and $\pi^* c(F/S)$ are smooth in $X'$  and so
there is a smooth form $v$ such that
$dd^c v=\pi^*c(F)-\pi^*c(S) \pi^*c(F/S)$.  By arguments as in the proof of Proposition~\ref{avig}
it follows that
$M^{g,a}$ is in the same class in $\mathcal B(X)$ as 
$$
\1_Z dd^c(\log|g|^2  c(F)/c(S))=c(F) \1_Z dd^c(\log|g|^2\sum_{\ell=0}^\infty \la dd^c\log|g|^2\ra^\ell)=c(F)\w M^g
=:M^{g,b}. 
$$
In particular, $M^{g,a}$ and $M^g$, as well as $M^{g,b}$, have the same multiplicities and fixed part. 
In case  $F$ has a trivial metric, these three currents coincide. 
\end{remark}

Let us conclude by mentioning two natural question that are not discussed in this paper.  The classical 
Poincar\'e-Lelong formula sometimes occurs in the form
$\dbar(1/g)\w D g/2\pi i=[\dv g],$
where $D$ is the Chern connection, which means that 
$$
\dbar\frac{1}{g}\w s(F) D g/2\pi i =M^g.
$$
Thus $M^g$ is a product of a residue current and a smooth form.  In a similar way the 
current  $M^{g,a}$ in Remark~\ref{apl}, see \cite[(6.4)]{Apl}, 
 can be written
$M^{g,a}=R^g \cdot \varphi,$
where $R^g$ is the Bochner-Martinelli residue current and 
$\varphi$ is a matrix of smooth forms involving both $Dg$ and the curvature tensor. 
We do not know whether there are analogues for $M^g$ even when $E$ is a line bundle.
Another natural question is whether some assumptions of positivity/negativity on $F$ and/or $E$ will imply positivity
of $M^g$;  see \cite{Apl} for some results of this kind for $M^{g,a}$.

 \bigskip

\noindent{\bf Acknowledment} The author is grateful to %
Richard L\"ark\"ang, H\aa kan Samuelsson Kalm and Elizabeth Wulcan for 
valuable discussions related to this paper, as well as to the referees for pointing out
several mistakes and obscurities.

\def\listing#1#2#3{{\sc #1}:\ {\it #2},\ #3.}

\end{document}